\documentclass{amsart}
\usepackage{amsmath}
\usepackage{amsthm}
\usepackage{amssymb}
\usepackage{bm}
\usepackage{mathrsfs}
\usepackage{graphicx}
\usepackage{fullpage}
\usepackage{graphicx}
\usepackage{xcolor}
\usepackage{tikz}

\usetikzlibrary{arrows}
\usetikzlibrary {positioning}

\newtheorem{theorem}{Theorem}[section]

\newtheorem{observation}[theorem]{Observation}

\newtheorem{lemma}[theorem]{Lemma}
\newtheorem{corollary}[theorem]{Corollary}
\theoremstyle{definition}
\newtheorem{definition}[theorem]{Definition}

\begin{document}
\title{Surrounding cops and robbers on graphs of bounded genus}

\author{Peter Bradshaw}
\address{Department of Mathematics, Simon Fraser University, Vancouver, Canada}
\email{peter\_bradshaw@sfu.ca}

\author{Seyyed Aliasghar Hosseini}
\address{School of Computing Science, Simon Fraser University, Vancouver, Canada}
\email{seyyed\_aliasghar\_hosseini@sfu.ca}

\thanks{The authors are partially supported by the Natural Sciences and Engineering Research Council of Canada (NSERC)}

\begin{abstract} We consider a surrounding variant of cops and robbers on graphs of bounded genus. We obtain bounds on the number of cops required to surround a robber on planar graphs, toroidal graphs, and outerplanar graphs. We also obtain improved bounds for bipartite planar and toroidal graphs. We briefly consider general graphs of bounded genus and graphs with a forbidden minor.
\end{abstract}
\maketitle

%\begin{lemma}
%In a Barnette graph, two vertices at distance $d$ along a face are at distance $d$. Furthermore, a path of length $d$ from one to the other belongs to that face.
%\end{lemma}

%\begin{theorem}
%$s(G)=k+1$ for the product of $k$ trees.
%\end{theorem}

%\begin{theorem}
%If you subdivide all edges constant number of times, $s(G)$ is reduced by at most half.
%\end{theorem}
\section{Introduction}
We consider the game of ``Surrounding Cops and Robbers," first invented by A. Burgess et. al. \cite{Burgess}. The rules of the game are as follows. The game is played by two players on a finite graph $G$ with perfect information. The first player controls a team of cops $C_1, \dots, C_m$, and the second player controls a robber. Initially, the first player places each cop $C_i$ on some vertex $v_i$ of $G$, and then the second player places the robber on some vertex $r$ of $G$. Then the two players alternate in taking turns. On the first player's turn, she moves each cop $C_i$ to an adjacent vertex or leaves $C_i$ at its current vertex. On the second player's turn, he moves the robber to an adjacent vertex or leaves the robber at its current vertex. Two cops may occupy the same vertex, and a cop and robber may occupy the same vertex. The first player wins if at any point the robber occupies a vertex $r$ such that each neighbor of $r$ is occupied by a cop; that is, the first player wins by ``surrounding" the robber with cops. The second player wins by indefinitely preventing the first player from winning; equivalently, the second player wins if any game position ever occurs twice. With these win conditions in mind, a final rule is added stating that the second player may not end a turn with the robber on the same vertex as a cop; this prevents the robber from remaining at a high-degree vertex forever. 

The game of Surrounding Cops and Robbers is one of several variants of the game of Cops and Robbers, invented by A. Quilliot \cite{Quilliot} and independently by R. Nowakowski and P. Winkler \cite{Nowakowski}. In the traditional game of Cops and Robbers, the player controlling the cops wins by moving a cop to the same vertex as the robber and thereby ``capturing" the robber. Several other Cops and Robbers variants exist. One such variant is Cops and Attacking Robber, invented by A. Bonato et. al. \cite{Bonato}, in which the robber may capture a cop by moving to the cop's vertex, after which the captured cop is removed from the game. Another variant is Lazy Cops and Robbers, invented by D. Offner and K. Ojakian \cite{Offner}, in which only one cop may move on each turn. One variant that shares remarkable similarities with Surrounding Cops and Robbers is a variant invented by M. Huggan and R. Nowakowski \cite{robot} in which the robber only decides his move after each cop has already committed to making a move.

When considering the game of Surrounding Cops and Robbers, one natural parameter to consider is the \textit{surrounding cop number}, which is defined in \cite{Burgess} as follows. Given a finite graph $G$, the surrounding cop number of $G$, denoted $s(G)$, is the minimum number $m$ such that if the first player has a team of $m$ cops, then the first player has a winning strategy in the game of Surrounding Cops and Robbers on $G$. The authors of \cite{Burgess} show several bounds on the surrounding cop number of certain graph classes, including grids, products of cycles, and normal Cayley graphs. The surrounding cop number is related to the notion of \textit{cop number}, introduced by M. Aigner and M. Fromme \cite{Aigner}, which is the minimum number of cops needed to give the player controlling the cops a winning strategy in the traditional game of Cops and Robbers.

The cop number of graphs of bounded genus is widely studied. Aigner and Fromme \cite{Aigner} show that planar graphs have cop number at most $3$; F. Lehner \cite{LehnerTorus} shows that toroidal graphs also have cop number at most $3$; and N. Bowler et. al. \cite{Lehner} show that graphs of genus $g$ have cop number at most $\lfloor \frac{4}{3}g + \frac{10}{3} \rfloor $. However, similar bounds for the surrounding cop number of graphs of bounded genus are not known. In this paper, we obtain surrounding cop number bounds for planar graphs, toroidal graphs, and outerplanar graphs. We also obtain improved bounds for bipartite planar and toroidal graphs, and we obtain some general bounds for graphs with forbidden minors. 
%{\PB I don't know how to format these citations. If there are no parentheses, then they distract from the numbers in the bounds, and if there are parentheses, it looks ugly.}

Our main tool will be the guarding of geodesic paths, introduced by Aigner and Fromme \cite{Aigner} and used by many other authors (c.f.  \cite{Lu}, \cite{QuilliotGenus},   \cite{Schroeder},  \cite{Scott}). 
%{\PB I don't know how important it is to list lots of papers that use the technique of guarding geodesics. Is this good? Or is it too many even?} 
Given a graph $G$, we will choose certain geodesic paths in $G$ to be \textit{guarded}. We will refer to the robber's \textit{region} or \textit{territory} as the component occupied by the robber in the graph obtained by removing guarded paths from $G$. We will successively make the robber's territory smaller until the robber's territory contains a single vertex, at which point we will show that the robber is surrounded. 
\section{Planar graphs}
In this section, we will consider the game of Surrounding Cops and Robbers on planar graphs. We will show that for planar graphs $G$, $s(G) \leq 7$. We will need some preliminaries. We will say that a $(u,v)$-path is a path with endpoints $u,v$.
\begin{definition}
Let $G$ be a graph. Let $P$ be a $(u,v)$-path of length $l$ in $G$. We say that $P$ is \textit{geodesic} with respect to $G$ if all $(u,v)$-paths in $G$ have length at least $l$.
\end{definition}
\begin{definition}
Let $G$ be a graph with a subgraph $H$. We say that $H$ is \textit{geodesically closed} with respect to $G$ if for any $u,v \in V(H)$, every geodesic $(u,v)$-path in $G$ is a subgraph of $H$.
\end{definition}

\begin{lemma}
Let $G$ be a graph, and let $P = (v_0, v_1, \dots, v_k)$ be a geodesic path in $G$. Suppose that a vertex $w \in V(G)$ is adjacent to $v_j \in V(P)$. Then $j-1 \leq dist(v_0,w) \leq j+1$.
\label{lemma3}
\end{lemma}
\begin{proof}
The path $(v_0, \dots, v_j, w)$ is a path of length $j+1$, so $dist(v_0,w) \leq j+1$. Suppose that $dist(v_0,w) \leq j-2$. Then $dist(v_0,v_j) \leq j-1$, implying $dist(v_0,v_k) \leq k-1$, which contradicts the assumption that $P$ is a geodesic path. 
\end{proof}

The following lemma shows that we may guard a geodesic path with three cops.

\begin{lemma}
Let $G$ be a graph. Let $(v_0, \dots, v_k)$ be a geodesic path of $G$. Then there exists a strategy using three cops such that after a finite number of moves, the robber is unable to access $P$.
\label{lemmaThreeCops}
\end{lemma}
\begin{proof}
We name our cops $C_1, C_2, C_3$. Whenever the robber occupies a vertex $w \in G$ with $dist(w,v_0) = d$ such that $d \leq k$, we say that the robber has a \textit{shadow} at the vertex $v_d$. Whenever the robber occupies a vertex $w \in G$ with $dist(w,v_0) \geq k$, we say that the robber has a shadow at $v_k$. When the robber moves, the distance from the robber to $v_0$ changes by at most $1$, and thus after each robber move, the robber's shadow on $P$ either stays put or moves to an adjacent vertex of $P$. Therefore, after a finite number of moves, we can reach one of the following types of positions:
\begin{itemize}
\item The robber's shadow occupies $v_i$ ($i \leq k-1$) and $C_1, C_2, C_3$ occupy $v_{i-1}, v_i,v_{i+1}$.
\item The robber's shadow occupies $v_k$, and $C_1, C_2, C_3$ occupy $v_{k-1}, v_k, v_{k}$.
\item The robber's shadow occupies $v_0$, and $C_1, C_2, C_3$ occupy $v_{0}, v_0, v_{1}$.
\end{itemize}

Furthermore, such a position can be achieved on each subsequent turn simply by following the robber's shadow with $C_1, C_2, C_3$. We call such a movement pattern \textit{stalking the robber's shadow} on $P$. We claim that when $C_1,C_2,C_3$ begin stalking the robber's shadow on $P$, the robber must leave $P$ for at least one turn. Indeed, if the robber occupies a vertex $v_j \in P$, then the robber's shadow occupies $v_j$, and $C_1, C_2, C_3$ occupy $v_j$, as well as all neighbors of $v_j$ on $P$; therefore, the robber must move off of $P$. 

Next, we show that after $C_1,C_2,C_3$ begin stalking the robber's shadow on $P$, the robber cannot enter $P$ from a vertex outside of $P$. Suppose that the robber occupies a vertex $w$ that does not belong to $P$. If $w$ is not adjacent to $P$, then the robber cannot move onto $P$. If $w$ is adjacent to $P$ and $dist(v_0,w) = j$, let $u \in P$ be a neighbor of $w$. By Lemma \ref{lemma3}, $u$ is a vertex that is occupied by one of $C_1,C_2,C_3$; therefore, the robber cannot move onto $P$.

We see that when $C_1, C_2, C_3$ begin stalking the robber's shadow on $P$, the robber is forced to exit $P$, and then the robber is never again able to enter $P$. Thus the lemma is proven.
\end{proof}

Next, we show that geodesically closed paths can be guarded with only two cops.
\begin{lemma}
Let $G$ be a graph. Let $P = (v_0, \dots, v_k) \subseteq G$ be a path. If $P$ is geodesically closed, then there exists a strategy involving two cops such that after a finite number of moves, then the robber is unable to access vertices of $P$.
\label{lemmaTwoCops}
\end{lemma}
\begin{proof}
We define the shadow of the robber as in Lemma \ref{lemmaThreeCops}. By the discussion in the proof of Lemma \ref{lemmaThreeCops}, two cops $C_1, C_2$ can reach vertices $v_{j-1}, v_j$, where $v_j$ is the position of the robber's shadow (or $C_1, C_2$ can reach $v_0, v_0$ if the robber's shadow occupies $v_0$) on every turn after a finite number of turns. Again, we call this pattern of movement \textit{stalking the robber's shadow on $P$}. We claim that by stalking the robber's shadow on $P$, $C_1$ and $C_2$ prevent the robber from entering $P$ from outside of $P$.

Suppose that the robber occupies a vertex $w \not \in V(P)$ that is adjacent to $v_j \in V(P)$. By Lemma \ref{lemma3}, $dist(v_0, w) \in \{j-1,j,j+1\}$. Furthermore, as $P$ is geodesically closed, $dist(v_0, w) \neq j-1$; otherwise, $(v_0, \dots, w, v_j)$ is a geodesic path between two vertices in $P$ that is not contained in $P$, a contradiction. Thus we see that $dist(v_0,w) \in \{j,j+1\}$, and thus the strategy of $C_1, C_2$ dictates that a cop occupy the vertex $v_j$. Therefore, the robber is unable to enter $P$.

It remains to show that $C_1,C_2$ can force the robber to leave the path $P$ by stalking the robber's shadow. Suppose that the robber occupies a vertex $v_j \in V(P)$ when $C_1, C_2$ begin stalking the robber's shadow. As the robber occupies $v_j$, the robber's shadow also occupies $v_j$, and $C_1, C_2$ occupy $v_{j-1},v_j$. (If the robber occupies $v_0$, then $C_1,C_2$ occupy $v_0,v_0$.) Thus the robber must move away from $v_j$. If the robber leaves $P$, then the proof is complete; otherwise, the robber moves to $v_{j+1}$. Then $C_1,C_2$ move to $v_j,v_{j+1}$, and the robber must move off of $P$ or move to $v_{j+2}$. By continuing to stalk the robber's shadow, the robber will either move off of $P$ voluntarily, or the robber will reach $v_k$, at which point $C_1, C_2$ will occupy $v_{k-1},v_k$. At this point, the robber will have no unoccupied neighbor in $P$, and the robber will be forced to leave $P$. This completes the proof.
\end{proof}

Our tools for guarding geodesic paths are in place. Now we will devise a strategy for using guarded geodesic paths to surround the robber. We will begin by enclosing the robber's region with two guarded paths. We will iteratively choose new paths to guard in order to make the robber's region smaller. Furthermore, we will always keep the robber's region enclosed by at least one path that requires only two cops to guard. Eventually, we will restrict the robber's region to a single vertex.  When the robber's region consists a single vertex $r$, every neighbor of $r$ in $G$ will belong to a guarded path. Furthermore, by construction, a path $P$ is guarded if and only if a set of cops moves along $P$ in such a way that whenever the robber occupies a vertex $v$ that is adjacent to $P$, all neighbors of $v$ in $P$ are occupied by cops. Therefore, when we say that each neighbor of $r$ belongs to a guarded path, this will imply that each neighbor of $r$ is occupied by a cop. Therefore, to show that we can surround the robber, we only need to show that we can reduce the robber's territory to a single vertex.

The following lemma shows that a path guarded by three cops can be exchanged for another guarded path using at most two extra cops.

\begin{lemma}
Let $G$ be a planar graph with a fixed drawing in the plane. Let $P_1, P_2 \subseteq G$ be two $(\alpha,\beta)$-paths in $G$, where $\alpha,\beta \in V(G)$. Let $A$ be a component of $G \setminus (P_1 \cup P_2)$ enclosed by $P_1 , P_2$. Suppose that for $i \in \{1,2\}$, $P_i$ is geodesic with respect to $P_i \cup A$. Suppose that the robber occupies a vertex in $A$. Suppose that $P_1$ is not geodesically closed with respect to $P_1 \cup A$, and three cops $C_1, C_2, C_3$ guard $P_1$ using the strategy in Lemma \ref{lemmaThreeCops}. Suppose that $P_2$ is guarded by two cops. Then there exists an $(\alpha,\beta)$-path $P_3 \subseteq A \cup P_1 $ such that $C_1, C_2, C_3$, and two other cops $C_4, C_5$ can guard $P_3$ without the robber reaching $P_1$, and such that after guarding $P_3$, the robber is restricted to a region $B \subsetneq A$. Furthermore, the cops can confine the robber to $B$ by guarding two $(\alpha,\beta)$-paths, at least one of which can be guarded with at most two cops.
\label{lemmaGenSwitch}
\end{lemma}
\begin{proof}
Let $P_1 = (\alpha = v_0, \dots, v_k = \beta)$. Whenever the robber occupies a vertex at distance $d$ from $v_0$, we say that the robber has a $P_1$-shadow at $v_d$. By the strategy in Lemma \ref{lemmaThreeCops}, $C_1,C_2,C_3$ stalk the robber's $P_1$-shadow on each turn. We may assume without loss of generality that on each turn, if the robber's $P_1$-shadow is at $v_d$ ($1 \leq d \leq k-1$), then $C_1$ occupies $v_{d-1}$, $C_2$ occupies $v_d$, and $C_3$ occupies $v_{d+1}$. If the robber's shadow occupies $v_k$, we may assume that $C_1$ occupies $v_{k-1}$ and $C_2,C_3$ occupy $v_k$. As $P_1$ is guarded, the robber's shadow cannot occupy $v_0$.

As $P_1$ is not geodesically closed in $A \cup P_1$, we can choose a geodesic (w.r.t. $A \cup P_1$) path $S \not \subseteq P_1$, with endpoints $v_i, v_j \in P$ $(i < j)$, such that $S$ is of shortest length out of all such geodesic paths. Let $P_1(i,j) = (v_i, v_{i+1}, \dots, v_j)$, and let $S = (v_i, w_{i+1}, \dots, w_{j-1}, v_j)$. Let $\aleph$ be the region enclosed by $S$ and $P_1(i,j)$. We may let $S$ be chosen out of all $(v_i,v_j)$-geodesics to minimize $\aleph$ and hence assume that $P_1 \cup \aleph$ does not contain any $(v_i,v_j)$-geodesic besides $P_1(i,j)$. (Such a path $S$ is called a \textit{bypath} in \cite{Sebastian}, in which S. Gonz\'{a}lez applies a similar technique to the active variant of cops and robbers.)

We claim that for any vertex $x \in A$, if $x$ is adjacent to a vertex $v_l \in P(i,j) \setminus \{v_i,v_j\}$, then $dist(v_0,x) \in \{l, l+1\}$. We know from Lemma \ref{lemma3} that $dist(v_0,x) \in \{l-1,l, l+1\}$. Suppose for the sake of contradiction that $dist(v_0,x) = l-1$. Then $x$ must belong to a geodesic path $S' = (v_m, \dots, x, v_l)$, where $v_m$ is chosen to make $S'$ as short as possible. By the minimality of $S$, $m < i$, and so $S'$ must cross $S$ at some vertex $w_p$. As $S$ and $S'$ are geodesics, it then follows that $(v_i, \dots, w_p, \dots, v_l)$ is a geodesic shorter than $S$, a contradiction. Thus the claim is proven.

We define $P_3 = (v_0, v_1, \dots, v_i, w_{i+1}, \dots, w_{j-1}, v_j, \dots, v_k)$. For convenience, we rename the vertices in $P_3$ so that $P_3 = (y_0, y_1, \dots, y_k)$. If the robber occupies a vertex at distance $d$ from $v_0$, then we say that the robber has a $P_3$-shadow at $y_d$. On each turn, the robber's $P_3$-shadow either stays put or moves to an adjacent vertex on $P_3$; hence, after a finite number of moves, we can place two cops $C_4,C_5$ at $y_{d-1},y_{d}$ (where $y_d$ is the position of the robber's $P_3$-shadow) on every turn. We execute such a strategy with $C_4, C_5$. Then, on each subsequent move after $C_4, C_5$ first reach the correct positions in relation to the robber's $P_3$-shadow, we move $C_3$ on $P_1$ toward the vertex $v_{j}$. We note that $C_3$ will reach $v_j$ before the robber's $P_1$-shadow can reach $v_j$. We let $C_1,C_2$ continue their original strategy on $P_1$.

We claim that if the robber's $P_3$-shadow ever occupies a vertex of $P_1$, then we have three cops in position to guard the path $P_3$ by the strategy in Lemma \ref{lemma3}. If the robber's $P_3$ shadow occupies a vertex of $P_1$, then one of two cases must have occured. \\
Case 1: The robber's $P_3$-shadow moves from $w_{i+1}$ to $v_i$. Then $C_1, C_2$ move to $v_{i-1}, v_i$, and $C_5$ remains at $w_{i+1}$.  The robber does not occupy a vertex of $P_3$, and $C_1,C_2,C_5$ stalk the robber's $P_3$-shadow; hence $C_1,C_2,C_5$ are guarding $P_3$ as in Lemma \ref{lemma3}. 

Case 2: The robber's $P_3$-shadow moves from $w_{j-1}$ to $v_j$. Then $C_3$ moves to $v_{j+1}$, and $C_4,C_5$ move to $w_{j-1},v_j$. The robber does not occupy a vertex of $P_3$, and $C_4,C_5,C_3$ stalk the robber's shadow on $P_3$; hence $C_4,C_5,C_3$ are guarding $P_3$ as in Lemma \ref{lemma3}. 

In both cases, three cops successfully guard $P_3$. If the robber's region is enclosed by $P_1,P_3$, then the robber's region is $\aleph$. $P_3$ is guarded by $3$ cops, and $P_1$ is geodesically closed with respect to $P_1 \cup \aleph$; hence the robber's region is enclosed by two $(\alpha,\beta)$-paths, one of which is guarded by only two cops $C_1,C_2$, and the lemma is proven. If the robber's region is enclosed by $P_2, P_3$, then we call the robber's region $B$. $P_3$ is guarded by $3$ cops, and $P_2$ is guarded by two cops by assumption, and the lemma is proven.

Hence we arrive at a point in which the robber's shadow must lie on the subpath $(w_{i+1},w_{i+2} \dots, w_{j-2}, w_{j-1})$. At this point, we move $C_3$ on $P_3$ toward the robber's shadow until $C_3$ and $C_4,C_5$ stalk the robber's shadow on $P_3$. At this point, one of two cases occurs: \\ \\
Case 1: The robber occupies a vertex in the region enclosed by $P_2$ and $P_3$. In this case, $P_2$ is guarded by two cops, and we have three cops $C_3,C_4,C_5$ guarding $P_3$. In this case, the proof is complete. \\ \\
Case 2: The robber occupies a vertex in the region $\aleph$ enclosed by $P_1$ and $P_3$. In this case, we have three cops $C_3,C_4,C_5$ guarding $P_3$. Additionally, $P_1$ is geodesically closed with respect to $P_1 \cup \aleph$, so $C_1,C_2$ can guard $P_1$ by continuing their current strategy. In this case, the proof is complete.

\end{proof}
With this lemma in place, we can prove an upper bound for the surrounding cop number of planar graphs. 
\begin{theorem}
Let $G$ be a planar graph. Then $s(G) \leq 7$.
\label{thmPlanar}
\end{theorem}
\begin{proof}
We will play the game in stages. In each Stage $i$ of the game, we will have two paths $P_1$ and $P_2$ with common endpoints that are guarded by cops in such a way that the robber is unable to access $P_1$ and $P_2$. We will define the robber's region $R_i$ at Stage $i$ as the component of $G \setminus (P_1 \cup P_2)$ that the robber occupies. We will show that at each stage, we can reduce the robber's region to a new region $R_{i+1} \subsetneq R_i$.

If $G$ is a tree, then by \cite{Burgess}, $s(G) = 2$, and we are done. Otherwise, we begin the game at Stage $1$. As $G$ is not a tree, there exists an edge $(uv) \in E(G)$ such that a geodesic path exists from $u$ to $v$ in the graph $G \setminus (uv)$.  We let $P_2$ be the path $(u,v)$, and we let $P_1$ be any $(u,v)$-path geodesic with respect to $G \setminus (uv)$. $P_2$ is of length one and can clearly be guarded with two cops. $P_1$ is a geodesic path with respect to $G \setminus P_2$, and so $P_1$ can be guarded with three cops by Lemma \ref{lemmaThreeCops}. At this point, the component of the graph $G \setminus (P_1 \cup P_2)$ that the robber occupies is called $R_1$. We may redraw $G$ so that $R_1$ is enclosed by $P_1, P_2$

We describe the strategy that we follow at each Stage $i$ of the game.  \\ \\
Case 1: The robber's region $R_i$ is enclosed by two guarded $(\alpha,\beta)$-paths $P_1,P_2$. $P_1$ is not geodesically closed with respect to $P_1 \cup R_i$, and $P_2$ is geodesically closed with respect to $P_2 \cup R_i$. 

Let $P_1$ be guarded by three cops $C_1,C_2,C_3$. Let $P_2$ be guarded by cops $C_4,C_5$. Then by Lemma \ref{lemmaGenSwitch}, we can use two extra cops $C_6,C_7$ and replace $P_1$ with an $(\alpha,\beta)$-path $P'_1$ so that at most three cops guard $P'_1$, and such that there exists an $(\alpha,\beta)$-path $P'_2$ guarded by at most two cops such that the robber is restricted to a component $R_{i+1}$ of $G \setminus (P'_1 \cup P'_2)$ with $R_{i+1} \subsetneq R_i$. Depending on whether or not $P_1'$ is geodesically closed with respect to $P_1' \cup R_{i+1}$, this brings us to Case 1, 2, or 3. \\ \\
Case 2: The robber's region $R_i$ is enclosed by two guarded $(\alpha,\beta)$-paths $P_1,P_2$. $P_1$ is geodesically closed with respect to $P_1 \cup R_i$, and $P_2$ is geodesically closed with respect to $P_2 \cup R_i$. There exists an $(\alpha,\beta)$-path in $P_1 \cup P_2 \cup R_i$ distinct from $P_1$ and $P_2$. 

We choose a shortest $(\alpha,\beta)$-path in $P_1 \cup P_2 \cup R_i$, and we call this path $P_3$. We guard $P_3$ with at most three cops. $P_3$ divides the region $R_i$ into two parts, and hence the robber's region is either enclosed by $P_1$ and $P_3$, or the robber's region is enclosed by $P_2$ and $P_3$. In both cases, the robber's region is restricted to a region $R_{i+1} \subsetneq R_i$, and one of the paths enclosing $R_{i+1}$ is guarded by at most two cops. Depending whether or not $P_3$ is geodesically closed with respect to $P_3 \cup R_{i+1}$, this brings us to Case 1, 2, or 3. \\ \\
Case 3: The robber's region $R_i$ is enclosed by two guarded $(\alpha,\beta)$-paths $P_1,P_2$. $P_1$ is geodesically closed with respect to $P_1 \cup R_i$, and $P_2$ is geodesically closed with respect to $P_2 \cup R_i$. There exists no $(\alpha,\beta)$-path in $P_1 \cup P_2 \cup R_i$ distinct from $P_1$ and $P_2$. 

In this case, without loss of generality, $P_1$ has only one vertex $x  \in V(P_1)$ adjacent to the robber's region, and $P_2$ is not adjacent to the robber's region. Then the robber can be restricted to his region simply by placing a cop at $x$. We thus place two cops $C_1, C_2$ at $x$ and confine the robber to his region. If $x$ has only one neighbor $y$ in the robber's region, then we move $C_1$ to $y$ and move $C_2$ to $y$ on the next turn. We may continue this process until $C_1, C_2$ occupy a vertex $x$ that has at least two neighbors $y,z$ in the robber's region $R_i$. Then we choose $P'_2 = (x,y)$, and we let $P'_1$ be an $(x,y)$-geodesic in $R_i \setminus (xy)$. We guard $P_1'$ with at most three cops and we may redraw $R_i$ so that the robber is enclosed by $P'_1,P'_2$. Then the robber is confined to a region $R_{i+1} \subsetneq R_i$, and depending on whether or not $P_1'$ is geodesically closed with respect to $P_1' \cup R_{i+1}$, this brings us to Case 1, 2, or 3.

If we continue this process, we will eventually reach a point in which the robber's region contains a single vertex, at which point the robber is surrounded.
\end{proof}

\begin{theorem}
There exists a planar graph $G$ with $s(G) \geq 6$.
\end{theorem}
\begin{proof}
Let $H$ be the truncated icosahedron, that is, the polyhedron whose shape resembles a soccer ball or a $C_{60}$ molecule. Let $G$ be the graph obtained by adding a vertex $v$ to each face $f$ of $H$ and adding an edge from $v$ to each vertex of $f$. We note that in $G$, all vertices have degree $5$ or $6$, and every degree $6$ vertex has five neighbors of degree $6$.

Suppose we have five cops. We let the robber execute the following strategy. The robber begins on a vertex $r$ of degree $6$ and waits for a cop to occupy $r$. As $r$ has degree $6$, and as we play with five cops, the robber will not be surrounded before a cop moves to $r$. If a cop moves to $r$, then at this point, at most four of the robber's neighboring vertices are occupied by cops. The robber then chooses an unoccupied neighbor of degree $6$ and moves to this vertex. The robber repeats this strategy indefinitely and wins. Therefore, $s(G) \geq 6$.
\end{proof}

Although our methods give an upper bound of $7$ for the surrounding cop number of planar graphs, we are currently unable to find examples of planar graphs with surrounding cop number equal to $7$. This leaves the question of whether $s(G) \leq 6$ for planar graphs $G$, or whether a planar graph $G$ with $s(G) = 7$ exists. 

\section{Bipartite planar graphs}
In this section, we consider the game of Surrounding Cops and Robbers played on bipartite planar graphs. We establish an upper bound of $4$ for the surrounding cop number of bipartite planar graphs using the ideas of the previous section. We also show that this upper bound is tight. FIrst, we will establish bipartite versions of the previous lemmas.

\begin{lemma}
Let $G$ be a bipartite graph, and let $P = (v_0, v_1, \dots, v_k)$ be a geodesic path in $G$. Suppose that a vertex $w \in V(G)$ is adjacent to $v_j \in V(P)$. Then either $dist(v_0, w) = j-1$ or $dist(v_0,w) = j+1$.
\label{lemmaPLM}
\end{lemma}
\begin{proof}
By Lemma \ref{lemma3}, $dist(v_0,w) \in \{j-1,j,j+1\}$. If $dist(v_0,w) = j$, then there exists a closed walk $(v_0, \dots, v_j, w, \dots, v_0)$ of length $2j+1$, which contradicts the assumption that $G$ is bipartite. Hence $dist(v_0, w) = j-1$ or $dist(v_0,w) = j+1$.
\end{proof}

\begin{lemma}
Let $G$ be a bipartite graph. Let $P = (v_0, \dots, v_k)$ be a geodesic path of $G$. Then there exists a strategy using two cops such that after a finite number of moves, the robber is unable to reach $P$.
\label{lemmaBipPath}
\end{lemma}
\begin{proof}
We name our cops $C_1, C_2$. We define the robber's shadow as in the proof of Lemma \ref{lemmaThreeCops}. Following the discussion in the proof of Lemma \ref{lemmaThreeCops}, after a finite number of moves, we can reach a position in which the robber's shadow occupies $v_i$ ($i \leq k-1$) and $C_1, C_2$ occupy $v_{i-1}, v_{i+1}$, or we can reach a position in which the robber's shadow occupies $v_k$, and $C_1, C_2$ occupy $v_{k-1}, v_{k}$, or we can reach a position in which the robber's shadow occupies $v_0$, and $C_1, C_2$ occupy $v_{0}, v_{1}$. Furthermore, such a position can be achieved on each subsequent turn simply by following the robber's shadow with $C_1, C_2$. We call such a movement pattern \textit{stalking the robber's shadow on $P$}. 

We show that if the robber does not occupy a vertex of $P$ when $C_1,C_2$ begin stalking the robber's shadow on $P$, then the robber is unable to enter $P$. Suppose that the robber occupies a vertex $w$ that does not belong to $P$. If $w$ is not adjacent to $P$, then the robber cannot move onto $P$. If $w$ is adjacent to $P$ and $dist(v_0,w) = j$, let $u \in P$ be a neighbor of $w$. By Lemma \ref{lemmaPLM}, $u  \in \{v_{j-1}, v_{j+1}\}$. At this point, the robber's shadow occupies $v_j$, and hence by the strategy of $C_1, C_2$, $u$ is occupied by a cop. Thus the robber cannot move onto $P$.

It remains to show that if the robber occupies a vertex of $P$ when $C_1, C_2$ begin stalking the robber's shadow on $P$, then the robber can be forced off of $P$, and that $C_1$ and $C_2$ can stalk the robber's shadow after the robber leaves $P$.

Suppose that the robber occupies a vertex $v_j \in P$. Then the robber's shadow also occupies $v_j$, and $C_1$ and $C_2$ occupy $v_{j-1}$ and $v_{j+1}$. Without loss of generality, $C_1$ occupies $v_{j-1}$ and $C_2$ occupies $v_{j+1}$. We let $C_1$ move to $v_j$ and let $C_2$ stay put. If the robber's shadow moves to either $v_j$ or $v_{j+1}$, then the robber must have left $P$, and $C_1$ and $C_2$ can stalk the robber's shadow using the original strategy. Otherwise, the robber's shadow moves to $v_{j-1}$. Then $C_1$ and $C_2$ move to $v_{j-1}$ and $v_j$. If the robber's shadow moves to $v_{j-1}$ or $v_j$, then by the same argument, $C_1$ and $C_2$ can stalk the robber's shadow on $P$ and prevent the robber from accessing $P$. Otherwise, the robber's shadow again moves toward $v_0$. In this case, $C_1$ and $C_2$ can continue pushing the robber's shadow toward $v_0$ until either the robber's shadow can be stalked by $C_1$ and $C_2$ as in the previous discussion, or until the robber's shadow reaches $v_0$. If the robber's shadow reaches $v_0$, then this implies that the robber occupies $v_0$. Then $C_1$ and $C_2$ will occupy $v_0$ and $v_1$, and the robber's shadow will be forced to move to $v_1$, and the robber will be forced to leave $P$. Then $C_2$ can move to $v_2$, and $C_1,C_2$ can stalk the robber's shadow and prevent the robber from accessing $P$. 
\end{proof}

\begin{lemma}
Let $G$ be a bipartite graph. Let $P = (v_0, \dots, v_k) \subseteq G$ be a path. If $P$ is geodesically closed with respect to $G$, then there exists a strategy involving one cop such that after a finite number of moves, the robber is unable to enter $P$ from outside of $P$.
\label{lemmaOneCop}
\end{lemma}
\begin{proof}
We define the shadow of the robber as in Lemma \ref{lemmaThreeCops}. By the discussion in the proof of Lemma \ref{lemmaThreeCops}, a single cop can reach a vertex $v_{j-1}$, where $v_j$ is the position of the robber's shadow, on every turn after a finite number of turns. We claim that by doing so, the cop prevents the robber from entering $P$ from outside of $P$.

Suppose that the robber occupies a vertex $w$ that is not in $P$ and is adjacent to $v_j \in P$. As $G$ is bipartite, either $dist(v_0, w) = j-1$ or $dist(v_0,w) = j+1$. Furthermore, as $P$ is geodesically closed, $dist(v_0, w) = j+1$; otherwise, $(v_0, \dots, w, v_j)$ is a geodesic path between two vertices in $P$ that is not contained in $P$, a contradiction. Thus we see that $dist(v_0,w) = j+1$, and the cop's strategy dictates that the cop occupy the vertex $v_j$. Therefore, the robber is unable to move to $v_j$, and hence the robber is unable to enter $P$.
\end{proof}

We will establish a path-switching lemma that is similar to Lemma \ref{lemmaGenSwitch}.

\begin{lemma}
Let $G$ be a planar bipartite graph with a fixed drawing in the plane. Let $P_1, P_2 \subseteq G$ be two $(\alpha,\beta)$-paths in $G$, where $\alpha,\beta \in V(G)$. Let $A$ be a component of $G \setminus (P_1 \cup P_2)$ enclosed by $P_1 \cup P_2$. Suppose that for $i \in \{1,2\}$, $P_i$ is geodesic with respect to $P_i \cup A$. Suppose that the robber occupies a vertex in $A$. Suppose that $P_1$ is not geodesically closed with respect to $P_1 \cup A$, and two cops $C_1, C_2$ are guarding $P_1$ using the strategy in Lemma \ref{lemmaBipPath}. Suppose that $P_2$ is guarded by one cop. Then there exists an $(\alpha,\beta)$-path $P_3 \subseteq A \cup P_1 $ such that $C_1, C_2$, and one other cop $C_3$ can guard $P_3$ without the robber reaching $P_1$, and such that after guarding $P_3$, the robber is restricted to a region $B \subsetneq A$. Furthermore, the cops can confine the robber to $B$ by guarding two $(\alpha,\beta)$-paths, at least one of which can be guarded with at most one cop.
\label{lemmaSwitch}
\end{lemma}
\begin{proof}

Let $P_1 = (\alpha = v_0, \dots, v_k = \beta)$. Whenever the robber occupies a vertex at distance $d$ from $v_0$, we say that the robber has a $P_1$-shadow at $v_d$. By the strategy in Lemma \ref{lemmaBipPath}, $C_1,C_2$ stalk the robber's $P_1$-shadow on each turn. We may assume without loss of generality that on each turn, if the robber's $P_1$-shadow is at $v_d$ ($1 \leq d \leq k-1$), then $C_1$ occupies $v_{d-1}$, and $C_2$ occupies $v_{d+1}$. If the robber's shadow occupies $v_k$, we may assume that $C_1$ occupies $v_{k-1}$ and $C_2$ occupies $v_k$. As $P_1$ is guarded, the robber's shadow cannot occupy $v_0$.

As $P_1$ is not geodesically closed in $A \cup P_1$, we can choose a geodesic (w.r.t. $A \cup P_1$) path $S \not \subseteq P_1,$ with endpoints $v_i, v_j \in P$ $(i < j)$ such that $S$ is of shortest length out of all such geodesic paths. Let $P_1(i,j) = (v_i, v_{i+1}, \dots, v_j)$, and let $S = (v_i, w_{i+1}, \dots, w_{j-1}, v_j)$. Let $\aleph$ be the region enclosed by $S$ and $P_1(i,j)$. We may let $S$ be chosen out of all $(v_i,v_j)$-geodesics to minimize $\aleph$ and hence assume that $P_1 \cup \aleph$ does not contain any $(v_i,v_j)$-geodesic besides $P_1(i,j)$.

We claim that for any vertex $x \in A$, if $x$ is adjacent to a vertex $v_l \in P(i,j) \setminus \{v_i,v_j\}$, then $dist(v_0,x) = l+1$. We know from Lemma \ref{lemmaPLM} that $dist(v_0,x) \in \{l-1, l+1\}$. Suppose for the sake of contradiction that $dist(v_0,x) = l-1$. Then $x$ must belong to a geodesic path $S' = (v_m, \dots, x, v_l)$, where $v_m$ is chosen to make $S'$ as short as possible. By the minimality of $S$, $m < i$, and so $S'$ must cross $S$ at some vertex $w_p$. As $S$ and $S'$ are geodesics, it then follows that $(v_i, \dots, w_p, \dots, v_l)$ is a geodesic shorter than $S$, a contradiction. Thus the claim is proven.

We define $P_3 = (v_0, v_1, \dots, v_i, w_{i+1}, \dots, w_{j-1}, v_j, \dots, v_k)$. For convenience, we rename the vertices in $P_3$ so that $P_3 = (y_0, y_1, \dots, y_k)$. If the robber occupies a vertex at distance $d$ from $v_0$, then we say that the robber has a $P_3$-shadow at $y_d$. On each turn, the robber's $P_3$-shadow either stays put or moves to an adjacent vertex on $P_3$; hence, after a finite number of moves, we can place a cop $C_3$ at $y_{d}$ (where $y_d$ is the position of the robber's $P_3$-shadow) on every turn. We execute such a strategy with $C_3$, and then on each subsequent move after $C_3$ first reaches the correct position in relation to the robber's $P_3$-shadow, we move $C_2$ on $P_1$ toward the vertex $v_j$. We note that $C_2$ will reach $v_j$ before the robber's $P_1$-shadow can reach $v_j$. We let $C_1$ continue as normal.

We claim that if the robber's $P_3$-shadow ever occupies a vertex of $P_1$, then we have two cops in position to guard the path $P_3$ by the strategy in Lemma \ref{lemmaBipPath}. Indeed, if the robber's $P_3$ shadow occupies a vertex of $P_1$, then one of two cases must have occured. \\ \\
Case 1: The robber's $P_3$-shadow moves from $w_{i+1}$ to $v_i$. Then $C_1$ moves to $v_{i-1}$, and $C_3$ remains at $w_{i+1}$.  The robber does not occupy a vertex of $P_3$, and $C_1,C_3$ stalk the robber's $P_3$-shadow; hence $C_1,C_3$ are guarding $P_3$ as in Lemma \ref{lemmaBipPath}.  \\ \\
Case 2: The robber's $P_3$-shadow moves from $w_{j-1}$ to $v_j$. Then $C_2$ moves to $v_{j+1}$, and $C_3$ remains at $w_{j-1}$. The robber does not occupy a vertex of $P_3$, and $C_2,C_3$ stalk the robber's shadow on $P_3$; hence $C_2,C_3$ are guarding $P_3$ as in Lemma \ref{lemmaBipPath}. 

In both cases, two cops successfully guard $P_3$. If the robber's region is enclosed by $P_1,P_3$, then the robber's region is $\aleph$. $P_3$ is guarded by two cops, and $P_1$ is geodesically closed with respect to $P_1 \cup \aleph$; hence the robber's region is enclosed by two paths, one of which is guarded by only one cop $C_1$, and the lemma is proven. If the robber's region is enclosed by $P_2, P_3$, then we call the robber's region $B$. $P_3$ is guarded by two cops, and $P_2$ is guarded by one cops by assumption, and the lemma is proven.

Hence we arrive at a point in which the robber's shadow must exist on the subpath $(w_{i+1},w_{i+2} \dots, w_{j-2}, w_{j-1})$. At this point, we move $C_2$ on $P_3$ toward the robber's shadow until $C_2$ and $C_3$ stalk the robber's shadow on $P_3$. At this point, one of three cases occurs: \\ \\
Case 1: The robber occupies a vertex in the region enclosed by $P_2$ and $P_3$. In this case, $P_2$ is guarded by one cop, and we have two cops $C_2,C_3$ guarding $P_3$. In this case, the proof is complete. \\ \\
Case 2: The robber occupies a vertex in the region $\aleph$ enclosed by $P_1$ and $P_3$. In this case, we have two cops $C_2,C_3$ guarding $P_3$. Additionally, $P_1$ is geodesically closed with respect to $P_1 \cup \aleph$, so $C_1$ can guard $P_1$ by continuing its current strategy. In this case, the proof is complete. \\ \\
Case 3: The robber occupies a vertex $w_l$ of $P_3$. In this case, $C_2$ occupies $w_{l-1}$, and $C_3$ occupies $w_{l+1}$. Then $C_2$ moves to $w_l$, and $C_3$ remains at $w_{l+1}$. $C_1$ continues its strategy. If the robber's $P_3$-shadow moves to $w_l$ or $w_{l+1}$, then the robber moves off of $P_3$, and $C_2$ and $C_3$ can stalk the robber's $P_3$-shadow and bring us to Case 1 or Case 2. Otherwise, the robber's $P_3$-shadow moves toward $v_i$. In this case, $C_2$ and $C_3$ both move along $P_3$ toward $v_i$ and push the robber's $P_3$-shadow toward $v_i$. By repeating this process, either the robber's $P_3$-shadow will make a movement that allows itself to be stalked, giving us Case 1 or Case 2, or the robber's $P_3$-shadow will reach $w_{i+1}$, with $C_1$ occupying $v_i$, $C_3$ occupying $w_{i+1}$, and $C_2$ occupying $w_{i+2}$. At this point, the robber must move to a vertex that is not on $P_3$, and the robber's $P_3$-shadow can be stalked on the next move, giving us Case 1 or Case 2. Thus in this case, the proof is complete. 
\end{proof}

With this lemma in place, we can prove an upper bound for the surrounding cop number of bipartite planar graphs. 

\begin{theorem}
Let $G$ be a planar bipartite graph. Then $s(G) \leq 4$.
\label{thmBip}
\end{theorem}

\begin{proof}
We will play the game in stages. In each Stage $i$ of the game, we will have two paths $P_1$ and $P_2$ with common endpoints that are guarded by cops in such a way that the robber is unable to access the vertices of $P_1$ and $P_2$. We will define the robber's region $R_i$ at Stage $i$ as the component of $G \setminus (P_1 \cup P_2)$ that the robber occupies. We will show that at each stage, we can reduce the robber's region to a new region $R_{i+1} \subsetneq R_i$.

If $G$ is a tree, then $s(G) = 2$, and we are done. Otherwise, we begin the game at Stage $1$. As $G$ is not a tree, there exists an edge $(uv) \in E(G)$ such that a geodesic path exists from $u$ to $v$ in the graph $G \setminus (uv)$.  We let $P_2$ be the path $(u,v)$, and we let $P_1$ be any $(u,v)$-path geodesic with respect to $G \setminus (uv)$. $P_2$ is geodesically closed with respect to $G$ and thus can be guarded by one cop. $P_1$ is a geodesic path with respect to $G \setminus P_2$, and so $P_1$ can be guarded with two cops by Lemma \ref{lemmaBipPath}. At this point, the component of the graph $G \setminus (P_1 \cup P_2)$ that the robber occupies is called $R_1$. We may redraw $G$ so that $R_1$ is enclosed by $P_1, P_2$.

We describe the strategy that we follow at each Stage $i$ of the game.  \\ \\
Case 1: The robber's region $R_i$ is enclosed by two guarded $(\alpha,\beta)$-paths $P_1,P_2$. $P_1$ is not geodesically closed with respect to $P_1 \cup R_i$, and $P_2$ is geodesically closed with respect to $P_2 \cup R_i$. 

Let $P_1$ be guarded by two cops $C_1,C_2$. Let $P_2$ be guarded by a cop $C_3$. Then by Lemma \ref{lemmaSwitch}, we can use one extra cop $C_4$ and replace $P_1$ with an $(\alpha,\beta)$-path $P'_1$ so that at most two cops guard $P'_1$, and such that there exists an $(\alpha,\beta)$-path $P'_2$ guarded by at most one cop such that the robber is restricted to a component $R_{i+1}$ of $G \setminus (P'_1 \cup P'_2)$ with $R_{i+1} \subsetneq R_i$. Depending on whether or not $P_1'$ is geodesically closed with respect to $P_1' \cup R_{i+1}$, this brings us to Case 1, 2, or 3. \\ \\
Case 2: The robber's region $R_i$ is enclosed by two guarded $(\alpha,\beta)$-paths $P_1,P_2$. $P_1$ is geodesically closed with respect to $P_1 \cup R_i$, and $P_2$ is geodesically closed with respect to $P_2 \cup R_i$. There exists an $(\alpha,\beta)$-path in $P_1 \cup P_2 \cup R_i$ distinct from $P_1$ and $P_2$. 

We choose a shortest $(\alpha,\beta)$-path in $P_1 \cup P_2 \cup R_i$, and we call this path $P_3$. We guard $P_3$ with at most two cops. $P_3$ divides the region $R_i$ into two parts, and hence the robber's region is either enclosed by $P_1$ and $P_3$, or the robber's region is enclosed by $P_2$ and $P_3$. In both cases, the robber's region is restricted to a region $R_{i+1} \subsetneq R_i$, and one of the paths enclosing $R_{i+1}$ is guarded by at most one cop. Depending whether or not $P_3$ is geodesically closed with respect to $P_3 \cup R_{i+1}$, this brings us to Case 1, 2, or 3. \\ \\
Case 3: The robber's region $R_i$ is enclosed by two guarded $(\alpha,\beta)$-paths $P_1,P_2$. $P_1$ is geodesically closed with respect to $P_1 \cup R_i$, and $P_2$ is geodesically closed with respect to $P_2 \cup R_i$. There exists no $(\alpha,\beta)$-path in $P_1 \cup P_2 \cup R_i$ distinct from $P_1$ and $P_2$. 

In this case, without loss of generality, $P_1$ has only one vertex $x  \in P_1$ adjacent to the robber's region, and $P_2$ is not adjacent to the robber's region. Then the robber can be restricted to his region simply by placing a cop at $x$. We thus place two cops $C_1, C_2$ at $x$ and confine the robber to his region. If $x$ has only one neighbor $y$ in the robber's region, then we move $C_1$ to $y$ and move $C_2$ to $y$ on the next turn. We may continue this process until $C_1, C_2$ occupy a vertex $x$ that has at least two neighbors $y,z$ in the robber's region $R_i$. Then we choose $P'_2 = (x,y)$, and we let $P'_1$ be an $(x,y)$-geodesic in $R_i \setminus (xy)$. We guard $P_1'$ with at most three cops and we may redraw $R_i$ so that the robber is enclosed by $P'_1,P'_2$. Then the robber is confined to a region $R_{i+1} \subsetneq R_i$, and depending on whether or not $P_1'$ is geodesically closed with respect to $P_1' \cup R_{i+1}$, this brings us to Case 1, 2, or 3.

If we continue this process, we will eventually reach a point in which the robber's region consists of a single vertex, at which point the robber is surrounded.
\end{proof}

Unlike Theorem \ref{thmPlanar}, we can show that the bound in Theorem \ref{thmBip} is tight.

\begin{theorem}
There exists a planar bipartite graph $G$ with $s(G) = 4$. 
\end{theorem}
\begin{proof}
Let $\ensuremath\mathcal{P}_5$ be the family of planar graphs with minimum degree $5$. Let $H \in \ensuremath\mathcal{P}_5$, and let $G$ be the graph obtained from $H$ by subdividing each edge exactly once. We color vertices originally from $H$ red, and we color vertices added as subdivisions blue. This is a proper coloring, and thus we see that $G$ is bipartite. 

We show that $3$ cops are not sufficient to surround the robber on $G$. The robber uses the following strategy. The robber begins at a red vertex $r \in G$ and does not move until a cop occupies $r$. As red vertices have degree at least $5$, the robber will not be surrounded before being captured. Suppose that a cop occupies $r$. As $H$ has minimum degree $5$, there are at least $5$ red vertices within distance $2$ of the robber. Let $v_1, \dots, v_5$ be $5$ such red vertices. The cop occupying $r$ does not have any of $v_1, \dots, v_5$ in its closed neighborhood. Furthermore, each of the other two cops does not have more than two of $v_1, \dots v_5$ in its closed neighborhood. Therefore, there exists a red vertex $v_i$ within distance $2$ of $r$ that is not in the closed neighborhood of any cop. At this point, the robber uses the next two moves to move on the shortest path toward $v_i$. As $v_i$ is not in the closed neighborhood of any cop, the robber will reach $v_i$ before any cop reaches $v_i$. The robber then can repeat this strategy indefinitely and avoid being surrounded forever. Hence $s(G) \geq 4$, and by Theorem \ref{thmBip}, $s(G) = 4$.
\end{proof}

\section{Toroidal graphs}
The aim of this section is to bound $s(G)$ for toroidal graphs $G$. We will mimic the strategy used by F. Lehner to show that every toroidal graph has cop number at most $3$ \cite{LehnerTorus}. Rather than considering the graph $G$, we will consider an infinite planar tiling $G^T$ of $G$, and we will attempt to capture the robber on $G^T$.  When the robber chooses a vertex $r$ of $G$, we will choose a corresponding vertex $r_0$ of $G^T$, and we will consider a large ball $B$ around $r_0$. We will then guard geodesic paths from $r_0$ to the boundary of $B$, and we will divide the robber's region in $B$ until the robber is restricted to a planar region of $G^T$. Then we will surround the robber by the planar strategy.

The following observation will be useful. 

\begin{observation}
In the proof of Theorem \ref{thmPlanar}, if we allow $P_2$ to be guarded by $3$ cops, then we have a strategy to capture the robber on a planar graph using $8$ cops.
\label{obs8}
\end{observation}

We will establish some preliminaries.

\begin{definition}
We say that a cyclic order of integers in the form $(a, a+1, \dots, a+m-1, a+m, a+m-1, \dots, a+1)$ is called a \textit{sawtooth order}.
\end{definition}

\begin{observation}
Let $A = (a, a+1, \dots, a+m-1, a+m, a+m-1, \dots, a+1)$ be a sawtooth order. Let $i, j \in A$, $i < j$. Then for any subsequence $B = (i, \dots, j) \subseteq A$, replacing $B$ with $(i, i+1, \dots, j-1, j)$ gives a sawtooth order.
\label{obsSaw}
\end{observation}

\begin{definition}
Let $H$ be a planar graph with an embedding in the plane, and let $v_0 \in V(H)$. Let $(h_1, h_2, h_3, \dots, h_m)$ be a clockwise walk around the boundary of $H$. Then we define the following cyclic order:
$$(H,v_0)_{bd}:=(dist(v_0,h_1), dist(v_0,h_2), dist(v_0,h_3), \dots, dist(v_0,h_m)).$$
\end{definition}

\begin{lemma}
Let $G$ be a planar graph with an embedding in the plane. Let $v_0,v_k \in V(G)$, and let $\Pi$ be the set of all geodesic paths from $v_0$ to $v_k$. There exist two geodesic paths $P_1, P_2$ from $v_0$ to $v_k$ enclosing an interior $\aleph$ such that $\bigcup \Pi \subseteq \aleph \cup P_1 \cup P_2$. 
\label{lemmaExt}
\end{lemma}
\begin{proof}
We will prove the following statement ($*$): \\ \\
Let $\Pi'$ be a proper subset of all geodesic paths from $v_0$ to $v_k$. Let $H \subseteq G$ be the union of all paths in $\Pi'$, and let $H$ inherit a planar embedding from $G$. Suppose that $(H,v_0)_{bd}$ is sawtooth. Then there exists a geodesic $(v_0,v_k)$-path $P' \not \in \Pi'$ such that $(H \cup P',v_0)_{bd}$ is sawtooth. ($*$) \\ \\
This statement ($*$) implies that we can let $\Pi'$ begin with a single geodesic path $P_0$, for which $(P_0,v_0)_{bd}$ is clearly sawtooth, and we can add new $(v_0,v_k)$-geodesics to $\Pi'$ one at a time until $\Pi'$ contains all $(v_0,v_k)$-geodesics---that is, until $\Pi' = \Pi$. Furthermore, in this way we can ensure that $(\bigcup \Pi, v_0)_{bd}$ is sawtooth. Then we let $(h_1, \dots, h_l, \dots, h_m)$ be a clockwise ordering of the boundary of $\bigcup \Pi$ such that $(dist(v_0, h_1), \dots, dist(v_0, h_l))$ is increasing and such that $(dist(v_0, h_l), \dots, dist(v_0,h_m))$ is decreasing. We let $q_1$ be a $(v_0, h_1)$-geodesic of $\bigcup \Pi$, and we let $q_1'$ be an $(h_m,v_k)$-geodesic of $\bigcup \Pi$. We let $q_2$ be a $(v_0, h_m)$-geodesic of $\bigcup \Pi$, and we let $q_2'$ be an $(h_l, v_k)$-geodesic of $\bigcup \Pi$. We see that $P_1: = q_1 \cup (h_1, \dots, h_l) \cup q_1'$ and $P_2:= q_2 \cup (h_m, \dots, h_l) \cup q_2'$ are $(v_0,v_k)$-geodesics of $G$ that enclose a region $\aleph$ such that $\bigcup \Pi \subseteq \aleph \cup P_1 \cup P_2$. Then the lemma is proven. Thus we aim to prove the statement ($*$).

 Let $\Pi'$ be a proper subset of all geodesic paths from $v_0$ to $v_k$. Let $H \subseteq G$ be the union of all paths in $\Pi'$, and let $H$ inherit a planar embedding from $G$. Suppose that $(H,v_0)_{bd}$ is sawtooth. Let $\aleph$ be the region including and enclosed by the boundary of $H$. Let $P^* \not \in \Pi'$ be a geodesic path from $v_0$ to $v_k$. We write $P^* = (w_0, \dots, w_k)$.

We make the following claim. Suppose that $w_i \in V(P^*)$ belongs to the boundary of $H$. Then $dist_H(v_0,w_i) = i$. To prove the lemma, suppose that $dist_H(v_0,w_i) < i$. Then there exists a path $P \subseteq H$ from $v_0$ to $w_i$ of length $i' < i$, implying that $P \cup P^*(w_i, v_k)$ is a walk from $v_0$ to $v_k$ of length less than $k$, a contradiction. Suppose, on the other hand, that $dist_H(v_0,w_i) > i$. Then there exists a path $P \subseteq H$ from $w_i$ to $v_k$ of length less than $k-i$, implying that $P^*(v_0, w_i) \cup P$ is a walk from $v_0$ to $v_k$ of length less than $k$, a contradiction. Thus we see that $dist_H(v_0,w_i) = i$.

If $P^* \subseteq \aleph$, then clearly $(H,v_0)_{bd}$ is sawtooth. Otherwise, $P^*$ is not contained in $\aleph$. Let $w_{i+1} \in P^*$ be the first vertex of $P^*$ that does not belong to $\aleph$, and let $w_j$ be the first vertex in $P^*$ after $w_{i+1}$ that belongs to $\aleph$. By the previous discussion, $dist_H(v_0,w_i) = i, dist_H(v_0,w_j) = j$. As $P^*$ is a geodesic, this implies that the subpath $q = (w_i, \dots, w_j)$ is of length $j-i$ and that there exist paths $p = (v_0, \dots, w_i) \subseteq H, p' = (w_j, \dots, v_k)$ respectively of lengths $i$ and $k-j$. Therefore, $P':= p \cup q \cup p'$ is a geodesic path in $G$ that is not included in $\aleph$. We show that $(H \cup P', v_0)_{bd}$ is sawtooth.

As $(dist(v_0,w_i), dist(v_0,w_{i+1}), \dots, dist(v_0,w_{j-1}), dist(v_0,w_j))$ = $(i,i+1, \dots, j-1, j)$, the cyclic order $(H \cup P', v_0)_{bd}$ is obtained from $(H,v_0)_{bd}$ by replacing a subsequence $(i, \dots, j)$ with $(i, i+1, \dots, j-1,j)$. By Observation \ref{obsSaw}, this leaves us with another sawtooth order. Hence $(H \cup P', v_0)_{bd}$ is also sawtooth.

Thus the statement ($*$) holds, and we see that there exist $(v_0,v_k)$-geodesic paths $P_1, P_2$ that enclose a region $\aleph$ such that $\bigcup \Pi \subseteq \aleph \cup P_1 \cup P_2$. Thus the lemma is proven.
\end{proof}

\begin{definition}
Let $G$ be an infinite graph. We say that $G$ has \textit{polynomial growth} if there exists a polynomial $f$ such that for any vertex $v \in G$, the number of vertices at distance exactly $d$ from $v$ is at most $f(d)$.
\end{definition}

The following lemma is proven in \cite{LehnerTorus}.

\begin{lemma}
Let $G$ be a finite toroidal graph. Then there is an infinite planar cover of $G$ with polynomial growth.
\label{lemmaPoly}
\end{lemma}

We are now ready to prove an upper bound for the surrounding cop number of toroidal graphs.

\begin{theorem}
Let $G$ be a toroidal graph. Then $s(G) \leq 8$.
\label{thmTorus}
\end{theorem}
\begin{proof}
Let $|G| = n$. We will essentially use the strategy of F. Lehner from \cite{LehnerTorus} with some slight modifications. Rather than considering $G$ directly, we will consider an infinite planar tiling $G^T$ of $G$ with polynomial growth (by Lemma \ref{lemmaPoly}). We note that there exists a natural projection function $\pi:V(G^T) \rightarrow V(G)$. If we play a game of cops and robbers on $G$, then at each point in the game, each cop $C$ has an infinite number of preimages in $G^T$ given by $\pi^{-1}(C)$. Furthermore, the robber $r$ has an infinite number of preimages $\pi^{-1}(r)$. The robber $r$ and a cop $C$ occupy the same vertex in $G$ if and only if some element of $\pi^{-1}(C)$ occupies the same vertex as some element of $\pi^{-1}(r)$ in $G^T$. Furthermore, the robber is surrounded on $G$ if and only if some element of $\pi^{-1}(r)$ occupies a vertex $x \in G^T$ such that all neighbors of $x$ are occupied by cop preimages. 

In our strategy, rather than aiming to let the cop preimages on $G^T$ surround any arbitrary robber preimage, we will focus on surrounding one predetermined robber preimage. This restriction can only make the game more difficult for the cops, and therefore any upper bound on the number of cops needed to surround a specific preimage of the robber on $G^T$ also gives an upper bound for the number of cops needed to surround the robber on $G$.

From this point onward, we will identify the robber with the $G^T$ preimage of the robber that we wish to surround. Let the robber begin the game at $r_0 \in G^T$. We let $D$ be a large value that is to be determined later. We define $B = B_{r_0}(D)$ as the ball of radius $D$ centered at $r_0$. Let $Q$ be the circumference of this ball; that is, let $Q$ be the set of vertices at distance exactly $D$ from $r_0$. We will need some lemmas.

\begin{lemma}
Let $v \in Q$. Let $P$ be a geodesic path in $B$ from $r_0$ to $v$. Suppose that the robber is at distance $d < D - 3n$ from $r_0$. Then $P$ can be guarded by three cop preimages as in Lemma \ref{lemmaThreeCops} before the robber reaches a distance of $d + 2n$ from $r_0$.
\label{lemma2n}
\end{lemma}
\begin{proof}
Let $P = (r_0, v_1, \dots,v_D = v )$. We define the shadow the robber on $P$ as in Lemma \ref{lemmaThreeCops}.

As the diameter of $G$ is at most $n$, three cops preimages $C_1, C_2, C_3$ can reach $v_{d+n+1},v_{d+n+2},v_{d+n+3}$ within $n$ moves. By assumption, when $C_1, C_2, C_3$ reach their positions, the shadow of the robber on $P$ is on the subpath $(r_0, v_1, \dots, v_{d+n})$. Then $C_1, C_2, C_3$ move along $P$ toward the robber's shadow until they reach a position to stalk the robber's shadow. $C_1,C_2,C_3$ reach such a position before the robber reaches $v_{d+2n}$.
\end{proof}

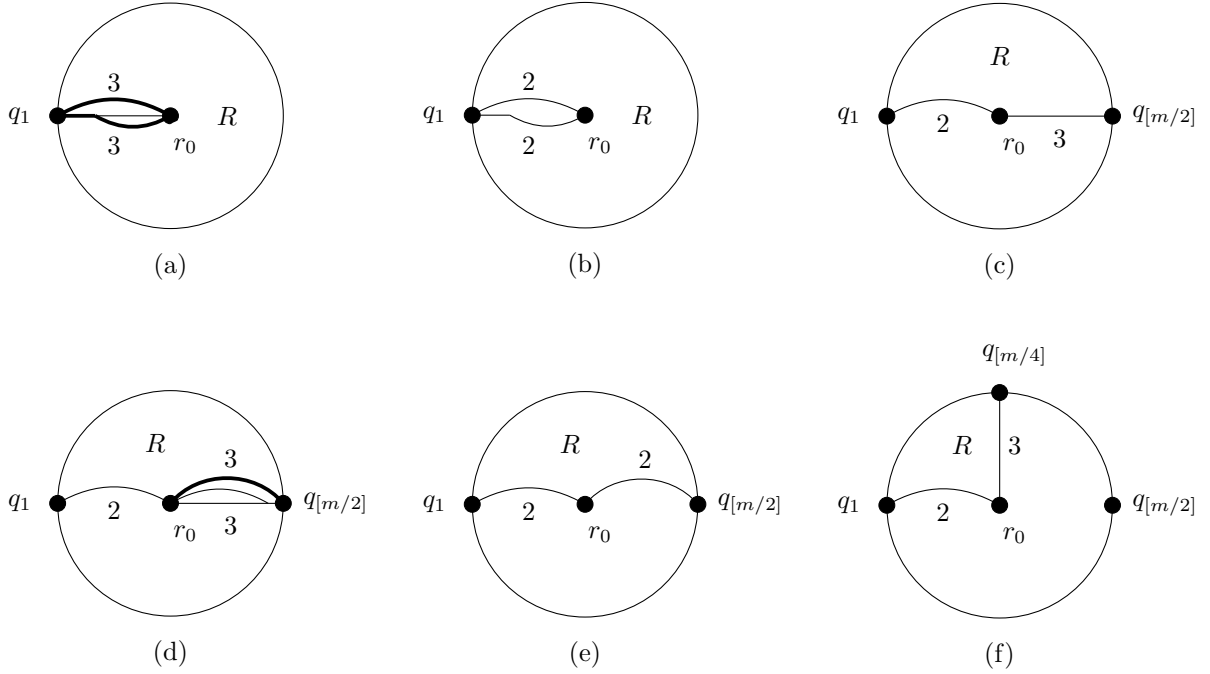
\begin{figure}
\begin{tabular} {l l l}
\begin{tikzpicture}
[scale=1,auto=left,every node/.style={circle,fill=gray!30}]
\draw (0,0) circle [radius=1.5];
\node (z) at (-0.75,0.45) [fill = white]  {$3$};
\node (z) at (-0.75,-0.4) [fill = white]  {$3$};
\node (z) at (0.75,0) [fill = white]  {$R$};
\node (z) at (0,-2) [fill = white]  {(a)};
\filldraw (-1.5,0) circle (3pt);
\filldraw (0,0) circle (3pt);
\draw (0,0)--(-1.5,0);
\draw [line width=0.5mm] (-1,0) -- (-1.5,0);
\node (z) at (-2,0) [fill = white]  {$q_1$};
\node (z) at (0.2,-0.4) [fill = white]  {$r_0$};
\draw [line width=0.5mm](0,0) to  [out=150,in=30] (-1.5,0) ;
\draw[line width=0.5mm] (0,0) to [out=-150,in=-30] (-1,0) ;
\end{tikzpicture} &
\begin{tikzpicture}
[scale=1,auto=left,every node/.style={circle,fill=gray!30}]
\draw (0,0) circle [radius=1.5];
\node (z) at (0.2,-0.4) [fill = white]  {$r_0$};
\node (z) at (0.75,0) [fill = white]  {$R$};
\node (z) at (-0.75,0.45) [fill = white]  {$2$};
\node (z) at (-0.75,-0.4) [fill = white]  {$2$};
\node (z) at (0,-2) [fill = white]  {(b)};
\filldraw (-1.5,0) circle (3pt);
\filldraw (0,0) circle (3pt);
\draw (-1,0)--(-1.5,0);
\node (z) at (-2,0) [fill = white]  {$q_1$};

\draw (0,0) to [out=150,in=30] (-1.5,0) ;
\draw (0,0) to [out=-150,in=-30] (-1,0) ;
\end{tikzpicture} &
\begin{tikzpicture}
[scale=1,auto=left,every node/.style={circle,fill=gray!30}]
\node (z) at (0.2,-0.4) [fill = white]  {$r_0$};
\node (z) at (2.2,0) [fill = white]  {$q_{[m/2]}$};
\node (z) at (0,-2) [fill = white]  {(c)};
\node (z) at (0,0.8) [fill = white]  {$R$};
\draw (0,0) circle [radius=1.5];
\node (z) at (-0.75,-0.1) [fill = white]  {$2$};
\node (z) at (0.8,-0.3) [fill = white]  {$3$};
\filldraw (-1.5,0) circle (3pt);
\filldraw (0,0) circle (3pt);
\filldraw (1.5,0) circle (3pt);

\draw (0,0)--(1.5,0);
\node (z) at (-2,0) [fill = white]  {$q_1$};
\draw (0,0) to [out=150,in=30] (-1.5,0) ;
\end{tikzpicture}
 \\
\begin{tikzpicture}
[scale=1,auto=left,every node/.style={circle,fill=gray!30}]
\node (z) at (0.2,-0.4) [fill = white]  {$r_0$};
\node (z) at (2.2,0) [fill = white]  {$q_{[m/2]}$};
\node (z) at (-0.2,0.8) [fill = white]  {$R$};
\node (z) at (0,-2) [fill = white]  {(d)};
\draw (0,0) circle [radius=1.5];
\node (z) at (0.8,-0.3) [fill = white]  {$3$};
\node (z) at (-0.75,-0.1) [fill = white]  {$2$};
\node (z) at (0.8,0.6) [fill = white]  {$3$};

\filldraw (-1.5,0) circle (3pt);
\filldraw (0,0) circle (3pt);
\filldraw (1.5,0) circle (3pt);

\draw  (0,0)--(1.5,0);
\node (z) at (-2,0) [fill = white]  {$q_1$};

\draw (0,0) to [out=150,in=30] (-1.5,0) ;

\draw (0,0) to [out=30,in=150] (1.3,0) ;
\draw [line width=0.5mm] (0,0) to [out=50,in=130] (1.5,0) ;

\end{tikzpicture} &
\begin{tikzpicture}
[scale=1,auto=left,every node/.style={circle,fill=gray!30}]
\node (z) at (0.2,-0.4) [fill = white]  {$r_0$};
\node (z) at (2.2,0) [fill = white]  {$q_{[m/2]}$};
\node (z) at (-0.2,0.8) [fill = white]  {$R$};
\draw (0,0) circle [radius=1.5];
\node (z) at (0,-2) [fill = white]  {(e)};
\node (z) at (0.8,0.6) [fill = white]  {$2$};
\node (z) at (-0.75,-0.1) [fill = white]  {$2$};

\filldraw (-1.5,0) circle (3pt);
\filldraw (0,0) circle (3pt);
\filldraw (1.5,0) circle (3pt);

\node (z) at (-2,0) [fill = white]  {$q_1$};

\draw (0,0) to [out=150,in=30] (-1.5,0) ;

\draw (0,0) to [out=50,in=130] (1.5,0) ;
\end{tikzpicture} &
\begin{tikzpicture}
[scale=1,auto=left,every node/.style={circle,fill=gray!30}]
\node (z) at (0.2,-0.4) [fill = white]  {$r_0$};
\node (z) at (-0.5,0.8) [fill = white]  {$R$};
\node (z) at (0.2,2) [fill = white]  {$q_{[m/4]}$};
\node (z) at (0,-2) [fill = white]  {(f)};
\node (z) at (0.2,0.8) [fill = white]  {$3$};
%\node (z) at (0.8,0.6) [fill = white]  {$2$};
\node (z) at (-0.75,-0.1) [fill = white]  {$2$};

\node (z) at (2.2,0) [fill = white]  {$q_{[m/2]}$};
\draw (0,0) circle [radius=1.5];
\filldraw (-1.5,0) circle (3pt);
\filldraw (0,0) circle (3pt);
\filldraw (0,1.5) circle (3pt);
\filldraw (1.5,0) circle (3pt);

\node (z) at (-2,0) [fill = white]  {$q_1$};

\draw (0,0) to [out=150,in=30] (-1.5,0) ;

%\draw (0,0) to [out=50,in=130] (1.5,0) ;

\draw (0,0)--(0,1.5);
\end{tikzpicture} \\
\end{tabular}
\caption{The figures show the initial maneuvers of the strategy of Theorem \ref{thmTorus}. In each figure, the circle represents the ball $B$ around $r_0$, and $R$ represents the robber's region in $B$. Figure (a) shows six cops guarding two $(r_0,q_1)$-geodesics that enclose all $(r_0,q_1)$-geodesics. Figure (b) shows that if R is not enclosed in the two geodesics guarded in (a), then these geodesics can be guarded with four cops.  Figure (c) shows $R$ being divided by an $(r_0,q_{[m/2]})$-geodesic. Figure (d) shows three additional cops guarding another $(r_0,q_{[m/2]})$-geodesic that separates $R$ from all $(r_0,q_{[m/2]})$-geodesics. Figure (e) shows that the paths adjacent to $R$ are geodesically closed with respect to $R$ and thus need only four cops to guard. Figure (f) shows $R$ being divided again by an $(r_0,q_{[m/4]})$-geodesic. The maneuvers shown in (d), (e), (f) can be repeated until the robber is contained in a planar region of $G^T$.}
\label{figTorus}
\end{figure}

Let $|Q| = m$. Let $(q_1, \dots, q_m)$ be the cyclic ordering of the elements of $Q$ according to the planar embedding of $G^T$. First, we use  Lemma \ref{lemmaExt} to compute two geodesics $P_1, P_2$ from $r_0$ to $q_1$ that enclose all ($r_0,q_1$)-geodesics. Then we use six cops to guard $P_1$ and $P_2$, as in Figure \ref{figTorus} (a). If the robber is enclosed by $P_1, P_2$, then the robber is restricted to a planar region by six cops, and then we can win the game with eight cops by Observation \ref{obs8}. Otherwise, $P_1$ and $P_2$ are geodesically closed with respect to the robber's territory in $B$, and we only need four of our six cops to continue to guard $P_1$ and $P_2$ (see Figure \ref{figTorus} (b)).

Next, we choose a geodesic $P_3$ from $r_0$ to $q_{[m/2]}$ and guard $P_3$ with $3$ cops (see Figure \ref{figTorus} (c)). Without loss of generality, the robber is restricted to a region of $B$ enclosed by $P_1$ and $P_3$, which are guarded by a total of $5$ cops.

From this point onward, we repeat the following procedure, which will recursively restrict the robber's territory until the robber is confined to a planar region of $G^T$.

Let the robber be confined to a region of $B$ bounded by geodesics $P_1 = (r_0, \dots, q_a)$ and $P_2 = (r_0, \dots, q_b)$. Suppose further that $P_2$ is guarded by two cops. Using Lemma \ref{lemmaExt}, we compute a geodesic $P_3$ from $r_0$ to $q_a$ such that $P_1$ and $P_3$ enclose all $(r_0, q_a)$-geodesics in the robber's territory (see Figure \ref{figTorus} (d)). We then use three cops to guard $P_3$. If the robber is in the interior of $P_1 \cup P_3$, then the robber is confined to a planar region of $B$ with six cops, and we win the game with eight cops by Observation \ref{obs8}. Otherwise, $P_3$ is geodesically closed with respect to the robber's territory and can be guarded by two cops. Now the robber is confined to a region of $B$ that is enclosed by $P_2$ and $P_3$, each of which is guarded by two cops (see Figure \ref{figTorus} (e)). Next, we use three cops to guard a geodesic $P_1'$ from $r_0$ to $q_{[(a+b)/2]}$ (see Figure \ref{figTorus} (f)). Now we see that, without loss of generality, the robber's territory in $B$ is enclosed by a geodesic $P_1'$ from $r_0$ to $q_{[(a+b)/2]}$ and a geodesic $P_2$ from $r_0$ to $q_b$. Furthermore, we see that $P_2$ is guarded by two cops. Thus the initial conditions of our procedure are satisfied again, and we can repeat our procedure.

We repeat this process until the robber's territory in $B$ is enclosed by a geodesic $P_1$ from $r_0$ to $q_a$ and a geodesic $P_2$ from $r_0$ to $q_{a+1}$. We see that reaching this point requires at most $\log m$ iterations of the procedure. As $(q_i)$ are in cyclic order with respect to the drawing of $G^T$, once $P_1$ and $P_2$ are guarded as such, the robber is confined to a planar region of $G^T$ and can be captured with eight cops by Observation \ref{obs8}.

As our procedure relies on Lemma \ref{lemma2n}, it remains only to show that the robber's distance from $r_0$ cannot reach $D - 3n$ before the procedure is complete. Any time a path is guarded, the robber is only able to move a distance of at most $2n$ away from $r_0$. To initialize our procedure, we require our cops to guard three paths. In each iteration of our procedure, we guard two paths. Furthermore, we execute at most $\log m = \log |Q|$ iterations of our procedure. Therefore, the total number of paths guarded in our strategy is at most $3 + 2 \log m$, and hence when our procedure is finished, the robber is at a distance of at most $6n + 4n \log m$ from $r_0$.

We now assign a value to $D$. We set $D = e^{kn}$, where $k$ is a sufficiently large constant. Recall that $m$ is bounded by the polynomial expression $f(m)$. We can bound the polynomial $f(m)$ by another polynomial $f^*(m)$ of the form $Am^{\alpha}$ such that $f(m) \leq f^*(m)$ for $m \geq 1$. Then the distance of the robber from $r_0$ is at most 
$$6n + 4n \log m \leq 6n + 4n \log (AD^{\alpha}) \leq 6n + 4n (\log A + \alpha kn) < e^{kn}-3n$$
for sufficiently large $k$. Therefore, we can choose $D$ large enough that the robber stays sufficiently far from the boundary of $B$ before being confined to a planar region of $G^T$. Finally, once the robber is confined to a planar region in $G^T$ with at most six cops, we can surround the robber with at most eight cops by Observation \ref{obs8}.
\end{proof}

\begin{theorem}
There exists a toroidal graph $G$ such that $s(G) \geq 7$.
\end{theorem}
\begin{proof}
Let $H$ be a $6$-regular tiling of equilateral triangles on the torus. Let $G$ be a graph obtained by adding a vertex $v$ at each face $f$ of $H$ and adding an edge from $v$ to every vertex of $f$. Clearly $G$ is toroidal. We call the vertices from $H$ \textit{original vertices}, and we call the additional vertices \textit{face vertices}. We let the robber play on the original vertices, each of which has degree $12$ in $G$. 

Suppose there are six cops. We note that the robber cannot be surrounded on an original vertex. Suppose that a cop moves to occupy $r$. At this point, at most five neighboring original vertices of $r$ are occupied by cops. Thus there exists an unoccupied original vertex adjacent to $r$, and the robber moves to this vertex. The robber repeats this process indefinitely and wins. Therefore, six cops are insufficient to surround the robber on $G$, and $s(G) \geq 7$.
\end{proof}

We can also use the methods in the proof of Theorem \ref{thmTorus} to prove an upper bound on the surrounding cop number of bipartite toroidal graphs. Furthermore, we can show that this upper bound is tight. 

\begin{theorem}
Let $G$ be a bipartite toroidal graph. Then $s(G) \leq 5$.
\label{thmBipTor}
\end{theorem}
\begin{proof}
The proof technique is nearly identical to that of Theorem \ref{thmTorus}. When $G$ is bipartite, however, paths that are geodesic with respect to the robber's region can be guarded with two cops, and paths that are geodesically closed with respect to the robber's region can be guarded with one cop. In Figure \ref{figBipTor}, we give a sketch of how the ideas of Theorem \ref{thmTorus} can be applied to a bipartite toroidal graph $G$ to restrict the robber to a planar region of the planar tiling $G^T$ using five cops. Then, by following the strategy of Theorem \ref{thmBip}, the robber can be surrounded using five cops.
\end{proof}

\begin{figure}
\begin{tabular} {l l l}
\begin{tikzpicture}
[scale=1,auto=left,every node/.style={circle,fill=gray!30}]
\draw (0,0) circle [radius=1.5];
\node (z) at (-0.75,0.45) [fill = white]  {$2$};
\node (z) at (0.2,-0.4) [fill = white]  {$r_0$};
\node (z) at (-0.75,-0.4) [fill = white]  {$2$};
\node (z) at (0.75,0) [fill = white]  {$R$};
\node (z) at (0,-2) [fill = white]  {(a)};
\filldraw (-1.5,0) circle (3pt);
\filldraw (0,0) circle (3pt);
\draw (0,0)--(-1.5,0);
\draw [line width=0.5mm] (-1,0) -- (-1.5,0);
\node (z) at (-2,0) [fill = white]  {$q_1$};

\draw [line width=0.5mm](0,0) to  [out=150,in=30] (-1.5,0) ;
\draw[line width=0.5mm] (0,0) to [out=-150,in=-30] (-1,0) ;
\end{tikzpicture} &
\begin{tikzpicture}
[scale=1,auto=left,every node/.style={circle,fill=gray!30}]
\node (z) at (0.2,-0.4) [fill = white]  {$r_0$};
\draw (0,0) circle [radius=1.5];
\node (z) at (0.75,0) [fill = white]  {$R$};
\node (z) at (-0.75,0.45) [fill = white]  {$1$};
\node (z) at (-0.75,-0.4) [fill = white]  {$1$};
\node (z) at (0,-2) [fill = white]  {(b)};
\filldraw (-1.5,0) circle (3pt);
\filldraw (0,0) circle (3pt);
\draw (-1,0)--(-1.5,0);
\node (z) at (-2,0) [fill = white]  {$q_1$};

\draw (0,0) to [out=150,in=30] (-1.5,0) ;
\draw (0,0) to [out=-150,in=-30] (-1,0) ;
\end{tikzpicture} &
\begin{tikzpicture}
[scale=1,auto=left,every node/.style={circle,fill=gray!30}]
\node (z) at (0.2,-0.4) [fill = white]  {$r_0$};
\node (z) at (2.2,0) [fill = white]  {$q_{[m/2]}$};
\node (z) at (0,-2) [fill = white]  {(c)};
\node (z) at (0,0.8) [fill = white]  {$R$};
\draw (0,0) circle [radius=1.5];
\node (z) at (-0.75,-0.1) [fill = white]  {$1$};
\node (z) at (0.8,-0.3) [fill = white]  {$2$};
\filldraw (-1.5,0) circle (3pt);
\filldraw (0,0) circle (3pt);
\filldraw (1.5,0) circle (3pt);

\draw (0,0)--(1.5,0);
\node (z) at (-2,0) [fill = white]  {$q_1$};

\draw (0,0) to [out=150,in=30] (-1.5,0) ;

\end{tikzpicture}
 \\
\begin{tikzpicture}
[scale=1,auto=left,every node/.style={circle,fill=gray!30}]
\node (z) at (0.2,-0.4) [fill = white]  {$r_0$};
\node (z) at (2.2,0) [fill = white]  {$q_{[m/2]}$};
\node (z) at (-0.2,0.8) [fill = white]  {$R$};
\node (z) at (0,-2) [fill = white]  {(d)};
\draw (0,0) circle [radius=1.5];
\node (z) at (0.8,-0.3) [fill = white]  {$2$};
\node (z) at (-0.75,-0.1) [fill = white]  {$1$};
\node (z) at (0.8,0.6) [fill = white]  {$2$};

\filldraw (-1.5,0) circle (3pt);
\filldraw (0,0) circle (3pt);
\filldraw (1.5,0) circle (3pt);

\draw  (0,0)--(1.5,0);
\node (z) at (-2,0) [fill = white]  {$q_1$};

\draw (0,0) to [out=150,in=30] (-1.5,0) ;

\draw (0,0) to [out=30,in=150] (1.3,0) ;
\draw [line width=0.5mm] (0,0) to [out=50,in=130] (1.5,0) ;

\end{tikzpicture} &
\begin{tikzpicture}
[scale=1,auto=left,every node/.style={circle,fill=gray!30}]
\node (z) at (2.2,0) [fill = white]  {$q_{[m/2]}$};
\node (z) at (0.2,-0.4) [fill = white]  {$r_0$};
\node (z) at (-0.2,0.8) [fill = white]  {$R$};
\draw (0,0) circle [radius=1.5];
\node (z) at (0,-2) [fill = white]  {(e)};
\node (z) at (0.8,0.6) [fill = white]  {$1$};
\node (z) at (-0.75,-0.1) [fill = white]  {$1$};

\filldraw (-1.5,0) circle (3pt);
\filldraw (0,0) circle (3pt);
\filldraw (1.5,0) circle (3pt);

\node (z) at (-2,0) [fill = white]  {$q_1$};

\draw (0,0) to [out=150,in=30] (-1.5,0) ;

\draw (0,0) to [out=50,in=130] (1.5,0) ;
\end{tikzpicture} &
\begin{tikzpicture}
\node (z) at (0.2,-0.4) [fill = white]  {$r_0$};
[scale=1,auto=left,every node/.style={circle,fill=gray!30}]
\node (z) at (-0.5,0.8) [fill = white]  {$R$};
\node (z) at (0.2,2) [fill = white]  {$q_{[m/4]}$};
\node (z) at (0,-2) [fill = white]  {(f)};
\node (z) at (0.2,0.8) [fill = white]  {$2$};
%\node (z) at (0.8,0.6) [fill = white]  {$1$};
\node (z) at (-0.75,-0.1) [fill = white]  {$1$};

\node (z) at (2.2,0) [fill = white]  {$q_{[m/2]}$};
\draw (0,0) circle [radius=1.5];
\filldraw (-1.5,0) circle (3pt);
\filldraw (0,0) circle (3pt);
\filldraw (0,1.5) circle (3pt);
\filldraw (1.5,0) circle (3pt);

\node (z) at (-2,0) [fill = white]  {$q_1$};

\draw (0,0) to [out=150,in=30] (-1.5,0) ;

%\draw (0,0) to [out=50,in=130] (1.5,0) ;

\draw (0,0)--(0,1.5);
\end{tikzpicture} \\
\end{tabular}
\caption{The figures show the initial maneuvers of the strategy of Theorem \ref{thmBipTor}. The figures have the same meanings as those of Figure \ref{figTorus}. In these figures, however, we let two cops guard a general path that is geodesic with respect to $R$, and we show one cop guarding a path that is geodesically closed with respect to $R$.}
\label{figBipTor}
\end{figure}
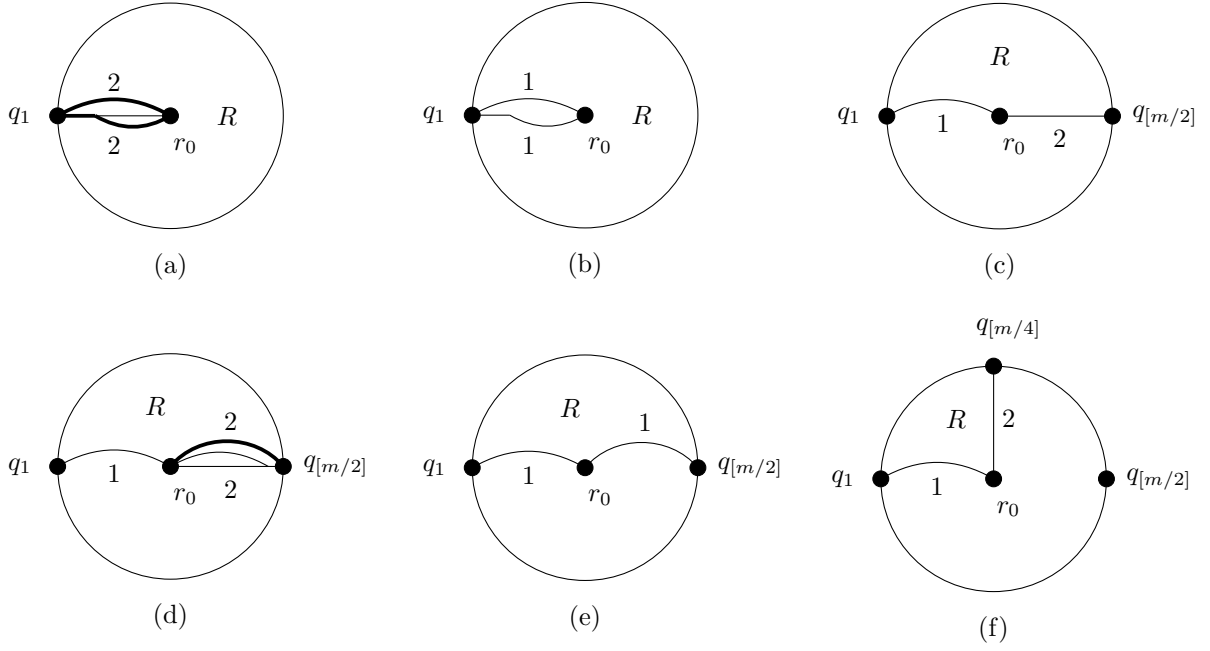

\begin{theorem}
There exists a bipartite toroidal graph $G$ with $s(G) = 5$.
\end{theorem}
\begin{proof}
Let $C_m, C_n$ be cycles respectively of length $m,n \geq 2$. Let $H = C_m \square C_n$, and let $H$ have a grid embedding on the torus. We construct a graph $G$ as follows: at each face $f$ of $H$, we add a quadrilateral, and we add an edge from each vertex of the quadrilateral to a vertex of $f$ in a way that does not introduce crossings. $G$ is clearly bipartite. We show that $s(G) = 5$.

By Theorem \ref{thmBipTor}, $s(G) \leq 5$. We call the vertices of $G$ that originate as vertices of $H$ \textit{original vertices.} We note that each original vertex is of degree $8$ and has four neighbors that are original vertices. We show that four cops are not sufficient to surround the robber. We let the robber begin at an original vertex $r$. As $r$ has degree $8$, robber cannot be surrounded without a cop first occupying $r$. If a cop occupies $r$, then at most three neighboring original vertices of $r$ are occupied by cops. Therefore, there exists an unoccupied neighboring original vertex of $r$, and the robber can move to this original vertex. The robber can repeat this process indefinitely. Therefore, exactly five cops are needed to surround the robber on $G$.
\end{proof}

\section{Outerplanar graphs}
In this section, we consider the game of Surrounding Cops and Robbers on outerplanar graphs. N.E. Clarke shows in \cite{Clarke} that the cop number of outerplanar graphs is at most $2$. In this section, we will show that for outerplanar graphs $G$, $s(G) \leq 3$.

The following lemma is given by G. Chartrand and F. Harary in \cite{Chartrand} and is fairly straightforward.
\begin{lemma}
Let $G$ be a two-connected outerplanar graph. Then $G$ is Hamiltonian, and in any outerplanar embedding of $G$, the facial walk of the exterior face of $G$ is a Hamiltonian cycle. 
\end{lemma}

We will prove that $s(G) \leq 3$ for two-connected outerplanar graphs $G$. Then we will use this to show that $s(G) \leq 3$ for all outerplanar graphs $G$.

For a graph $G$ with an outerplanar embedding, we will say that an edge that is adjacent to the exterior face of $G$ is called an \textit{exterior edge}. We will say that any edge that is not an exterior edge is an \textit{interior edge.}

\begin{lemma}
Let $G$ be a two-connected outerplanar graph with a fixed outerplanar embedding. Let $(v_0, \dots, v_{n-1})$ be the Hamiltonian cycle given by the exterior facial walk of $G$. Let $v_iv_j$ be an interior edge of $G$. Then $G \setminus \{v_i,v_j\}$ has two components, which are induced by $(v_{i+1}, \dots, v_{j-1})$ and $(v_{j+1}, \dots, v_{i-1})$ (where addition is considered modulo $n$).
\label{lemmaEdge}
\end{lemma}
\begin{proof}
Clearly the vertices $\{v_{i+1}, \dots, v_{j-1}\}$ belong to a single component of $G \setminus \{v_i,v_j\}$, as do the vertices $\{v_{j+1}, \dots, v_{i-1}\}$. Note that neither of these sets is empty, as $v_i, v_j$ are not adjacent on the outer face of $G$. Furthermore, the vertices $(v_i, v_{i+1}, \dots, v_{j-1}, v_j,v_i)$ form a cycle $C$. As $G$ is outerplanar, it follows that any non-exterior edge adjacent to a vertex of $\{v_{i+1}, \dots, v_{j-1}\}$ must be drawn in the interior of $C$ and therefore cannot have an endpoint that does not belong to $C$. We hence see that $\{v_{i+1}, \dots, v_{j-1}\}$ is a maximal connected component of $G \setminus \{v_i,v_j\}$, as is $\{v_{j+1}, \dots, v_{i-1}\}$ by a similar argument. This proves the lemma.
\end{proof}
\begin{corollary}
Let $G$ be a two-connected outerplanar graph with a fixed outerplanar embedding. Let $(v_0, \dots, v_{n-1})$ be the Hamiltonian cycle given by the exterior facial walk of $G$. Suppose that two cops occupy vertices $v_i,v_j$, where $v_iv_j \in E(G)$. Then the robber is restricted to either the vertex set $\{v_{i+1}, \dots, v_{j-1}\}$ or the vertex set $\{v_{j+1}, \dots, v_{i-1}\}$.
\end{corollary}
\begin{proof}
If $v_iv_j$ is an exterior edge of $G$, then without loss of generality, $\{v_{i+1}, \dots, v_{j-1}\}$ contains all vertices of $G \setminus \{v_i,v_j\}$, and clearly the robber is restricted to $G \setminus \{v_i,v_j\}$. If $v_iv_j$ is an interior edge of $G$, then by Lemma \ref{lemmaEdge}, the robber is restricted to a component of $G \setminus \{v_i, v_j\}$, and the result follows.
\end{proof}

The following lemma shows that in a two-connected outerplanar graph, the robber's region can be reduced using three cops.

\begin{lemma}
Let $G$ be a two-connected outerplanar graph with a fixed outerplanar embedding. Let $(v_0, \dots, v_{n-1})$ be the Hamiltonian cycle given by the exterior facial walk of $G$. Let $v_iv_j$ be an edge of $G$, and suppose that two cops occupy $v_i$ and $v_j$, restricting the robber to vertex set $X = \{v_{i+1}, \dots, v_{j-1}\}$. Then there exists a strategy involving $3$ cops by which two cops guard adjacent vertices $v_k, v_l$ and restrict the robber to a vertex set $\{v_{k+1}, \dots, v_{l-1}\} \subsetneq X$.
\label{lemmaSmallerPath}
\end{lemma}
\begin{proof}
We call the cop at $v_i$ $C_1$, and we call the cop at $v_j$ $C_2$. Let $H = G[v_i, v_{i+1}, \dots, v_{j-1}, v_j] \setminus (v_iv_j)$. We note that the robber's territory in $G$ is a subgraph of $H$. Let $P = (u_0, \dots, u_q)$ be a $(v_i,v_j)$-geodesic in $H$. Note that $q \geq 2$. We define the robber's shadow on $P$ as in Lemma \ref{lemmaThreeCops}. As before, the robber's shadow must either stay put in $P$ on each turn or move to an adjacent vertex of $P$. Therefore, a cop $C_3$ can move to the same vertex as the robber's shadow on every turn after a finite number of moves. We let $C_3$ execute such a strategy.

After $C_3$ begins ``capturing" the robber's shadow on every move, we let $C_1$ move toward the robber's shadow on $P$ on each turn until reaching $u_{k-1}$, where $u_k$ is the position of the robber's shadow. $C_1$ then moves to $u_{k-1}$ on each subsequent turn, where $u_k$ is a subsequent position of the robber's shadow. When $C_1$ executes this strategy, the robber will not be able to reach $v_i = u_0$; in order to reach $u_0$, the robber would first have to move to a vertex at distance one from $u_0$, at which point $C_1$ would occupy $u_0$. At the same time, we let $C_2$ move toward the robber's shadow on $P$ until reaching $u_{k+1}$, where $u_k$ is the position of the robber's shadow (or until reaching $u_q$ if the robber's shadow occupies $u_q$). $C_2$ then moves to $u_{k+1}$ on each subsequent turn, where $u_k$ is a subsequent position of the robber's shadow (or to $u_q$ if the robber's shadow occupies $u_q$). By a similar argument, the robber will not be able to reach $v_j$ while $C_2$ executes this strategy. Therefore, when $C_1$ and $C_2$ execute this strategy, the robber's territory in $G$ does not increase.

After $C_1, C_2, C_3$, successfully execute their strategies, $C_1,C_2,C_3$ guard the path $P$ by Lemma \ref{lemmaThreeCops}, and the robber is restricted to a component of $H \setminus \{u_0, \dots, u_q\}$. For $0 \leq i \leq q-1$, let $P(u_i,u_{i+1})$ be the graph induced by the vertices of the exterior path from $u_i$ to $u_{i+1}$ that does not include any vertex $u_j$ for $j \not \in \{i,i+1\}$. As $q \geq 2$, $P(u_i,u_{i+1})$ is uniquely defined.

If $H \setminus \{u_0, \dots, u_q\}$ is empty, then the robber has no territory, implying that the robber is surrounded by $C_1,C_2,C_3$. Otherwise, by Lemma \ref{lemmaEdge}, the components of the graph $H \setminus \{u_0, \dots, u_q\}$ must be of the form $P(u_i,u_{i+1}) \setminus \{u_i,u_{i+1}\}$. Therefore, the robber is restricted to a region $P(u_i,u_{i+1}) \setminus \{u_i,u_{i+1}\}$. As $P$ is guarded, we can move two cops to occupy $u_i,u_{i+1}$ before the robber can reach either of $u_i,u_{i+1}$. We move two cops in such a way to $u_i, u_{i+1}$. The vertices $u_i,u_{i+1}$ can be written as $v_k,v_l$. Furthermore, with $v_k, v_l$ guarded, the robber is restricted to a vertex set $\{v_{k+1}, \dots, v_{l-1}\}$, which is a proper subset of $\{v_{i+1}, \dots, v_{j-1}\}$. Thus the lemma is proven.
\end{proof}

\begin{corollary}
Let $H$ be a two-connected subgraph of an outerplanar graph $G$. Suppose that two cops occupy adjacent vertices $v_i,v_j \in H$. Then there exists a strategy involving three cops that removes $H$ from the territory of the robber.
\label{2con}
\end{corollary}
\begin{proof}
By following the strategy in Lemma \ref{lemmaSmallerPath}, we can iteratively reduce the number of vertices in $H$ that belong to the robber's territory. We can continue this process until no vertex of $H$ belongs to the robber's territory. Then the robber is either surrounded or prevented from entering $H$.
\end{proof}

\begin{theorem}
Let $G$ be an outerplanar graph. Then $s(G) \leq 3$.
\label{thmOut}
\end{theorem}
\begin{proof}
If $G$ is a tree, then by \cite{Burgess}, $s(G) \leq 2$. Otherwise, we name our cops $C_1,C_2,C_3$. We begin the game by choosing a maximal two-connected subgraph $H \subseteq G$. We place two cops at two adjacent vertices $v_i,v_j \in H$. We show that at each point of the game, we can reduce the robber's territory $R_i$. We consider two cases: \\ \\
Case 1: No two cops occupy the endpoints of an edge with both endpoints adjacent to the robber's territory, and a cop occupies a cut-vertex adjacent to the robber's territory. 

Without loss of generality, let $C_1$ occupy a cut-vertex $x$ adjacent to the robber's territory. We then let $C_2$ guard a neighbor $y$ of $x$ for which $y$ belongs to the robber's territory. This reduces the robber's territory to $R_{i+1} \subsetneq R_i$ and depending on whether or not $y$ is a cut-vertex of $G$, this brings us to Case 1 or 2. \\ \\
Case 2: Two cops occupy the endpoints of an edge with endpoints $u,v$ adjacent to the robber's territory. 

Let $H \subseteq G$ be a maximal two-connected subgraph containing $u,v$. By Corollary \ref{2con}, $C_1,C_2,C_3$ have a strategy to remove all vertices of $H$ from the robber's territory. Furthermore, as $u,v$ are both adjacent to the robber's territory, the robber's territory has at least one vertex in $H$. Therefore, by removing the vertices of $H$ from the robber's territory, we reduce the robber's territory to a region $R_{i+1} \subsetneq R_i$. If the robber is not surrounded during the execution of such a strategy, then the robber is forced to leave $H$, and there exists a cut-vertex $x \in H$ adjacent to the robber's territory. As the vertices of $H$ are guarded, the robber can be prevented from accessing $x$; that is, a cop can reach $x$ before the robber. We let a cop move to $x$ before the robber reaches $x$, and this brings us to Case 1.

By repeatedly reducing the robber's territory in this way, we will eventually reach a point in which the robber's territory is a single vertex. At this point, the robber is surrounded, and the cops win the game.
\end{proof}

Finally, we show that this bound is tight, even for bipartite outerplanar graphs. 

\begin{theorem}
There exists a bipartite outerplanar graph $G$ with $s(G) = 3$.
\end{theorem}
\begin{proof}
Let $G = P_1 \square P_3$ be the grid with $8$ vertices. We note that $G$ is bipartite and outerplanar. We show that $2$ cops cannot surround a robber on $G$.

We note that $G$ has four degree $3$ vertices that form a $4$-cycle $C$. We let the robber begin the game at a vertex $r$ of $C$. As the vertices of $C$ have degree $3$, the robber cannot be surrounded at $r$ without a cop moving to occupy $r$. If a cop occupies $r$, then at most one of the robber's $C$-neighbors is occupied by a cop, and thus the robber can move to a neighboring vertex in $C$. The robber can repeat this strategy indefinitely. Therefore, $s(G) \geq 3$. By Theorem \ref{thmOut}, it follows that $s(G) = 3$.
\end{proof}

\section{Graphs of higher genus and graphs that exclude a minor}
We will briefly consider the surrounding cop number bounds of graphs of higher genus and graphs that exclude a minor. In the traditional game of cops and robbers, strategies for capturing a robber on a graph of higher genus are similar to strategies for planar and toroidal graphs; that is, the cops capture the robber by guarding geodesic paths and iteratively reducing the robber's region, as in \cite{QuilliotGenus} and \cite{Schroeder}. Currently, the best known general strategy for capturing the robber on a graph of genus $g$ is given by N. Bowler et. al., who prove the following theorem.

\begin{theorem} \cite{Lehner}
Let $G$ be a graph of genus $g$. Then there exists a strategy to reduce the robber's region on $G$ to one vertex by guarding geodesic paths, in which no more than $\lfloor \frac{4}{3}g + \frac{10}{3} \rfloor$ geodesic paths are guarded at one time.
\label{thmG}
\end{theorem}

This immediately gives us the following result.
\begin{theorem}
Let $G$ be a graph of genus $g$. Then $s(G) \leq 4g + 10$. If $G$ is bipartite, then $s(G) \leq \lfloor \frac{8}{3}g + \frac{20}{3}\rfloor$.
\label{thmGenGenus}
\end{theorem}
\begin{proof}
We apply the strategy of Theorem \ref{thmG}, using three cops to guard a geodesic path in the general case, and using two cops to guard a geodesic path if $G$ is bipartite.
\end{proof}
More generally, we may also consider families of graphs that exclude a minor. A theorem of T. Andreae from \cite{Andreae} shows that if $G$ is a graph that does not contain $H$ as a minor, then for any vertex $h \in V(H)$ that is not adjacent to a leaf of $H$, $c(G) \leq |E(H - h)|$. Furthermore, the strategy of T. Andreae is carried out solely by guarding geodesic paths. Therefore, by using three cops to guard each path in T. Andreae's strategy (or two cops for the bipartite case), the strategy can be adapted to the surrounding variant of cops and robbers to give us the following theorem.
\begin{theorem}
Let $H$ be a graph, and let $h \in V(H)$ be a vertex of $H$ that has no neighbor of degree one. If $G$ is a graph that does not contain $H$ as a minor, then $s(G) \leq 3|E(H - h)|$. If $G$ is also bipartite, then $s(G) \leq 2|E(H-h)|$.
\end{theorem}

\section{Open questions}
We have shown that for planar graphs $G$, $s(G) \leq 7$, and $s(G)$ may be as large as $6$. Furthermore, for toroidal graphs $G$, we have shown that $s(G) \leq 8$, and $s(G)$ may be as large as $7$. Thus we may naturally ask the following questions:  
\begin{itemize}
\item Does there exist a planar graph $G$ for which $s(G) = 7$?
\item Does there exist a toroidal graph $G$ for which $s(G) = 8$?
\end{itemize}
The authors conjecture that both of these questions have a negative answer. 

Furthermore, S. Gonz\'{a}lez and B. Mohar show that three cops are enough to capture a robber on a planar graph even when only two cops are allowed to move on each turn \cite{Sebastian}. These authors also show that four cops are sufficient to capture a robber on a planar graph even when each cop is required to move on every turn \cite{SebastianThesis}. We may thus ask the following questions as well:
\begin{itemize}
\item How many cops are required to surround a robber on a planar (toroidal) graph when at most $k$ cops may move on each turn?
\item How many cops are required to surround a robber on a planar (toroidal) graph when all cops are required to move on each turn?
\end{itemize}

\section{Acknowledgments} 
The authors thank Sebasti\'{a}n Gonz\'{a}lez Hermosillo de la Maza for fruitful discussions. The authors also thank Florian Lehner for making several helpful corrections to an earlier version of this paper. Finally, the authors thank Bojan Mohar for several suggestions that improved the quality of the paper.

\raggedright
\bibliographystyle{plain}
\bibliography{SurroundingBib}

\end{document}